\documentclass{amsart}
\usepackage{amssymb,amscd}
\usepackage[cp1251]{inputenc}
\usepackage[T2A]{fontenc}

\newtheorem{theorem}{Theorem}[section]
\newtheorem{proposition}[theorem]{Proposition}
\newtheorem{corollary}[theorem]{Corollary}
\newtheorem{lemma}[theorem]{Lemma}

\theoremstyle{definition}
\newtheorem{definition}[theorem]{Definition}
\newtheorem{example}[theorem]{Example}
\newtheorem{construction}[theorem]{Construction}
\theoremstyle{remark}
\newtheorem*{remark}{Remark}

\numberwithin{equation}{section}
\numberwithin{figure}{section}

\def\C{\mathbb C}
\def\D{\mathbb D}
\def\I{\mathbb I}
\def\Q{\mathbb Q}
\def\R{\mathbb R}
\def\T{\mathbb T}
\def\Z{\mathbb Z}

\newcommand{\cF}{\mathcal F}
\def\sK{\mathcal K}

\def\phi{\varphi}

\newcommand{\mb}[1]{{\textbf {\textit#1}}}

\renewcommand{\ge}{\geqslant}
\renewcommand{\le}{\leqslant}

\newcommand{\Ker}{\mathop{\rm Ker}}
\def\rk{\mathop{\mathrm{rk}}}
\renewcommand{\Re}{\mathop{\mathrm{Re}}}
\renewcommand{\Im}{\mathop{\mathrm{Im}}}

\newcommand{\zk}{\mathcal Z_{\mathcal K}}
\newcommand{\zp}{\mathcal Z_P}

\def\Gr{\mathop{\mathrm{Gr}}}

\begin{document}

\title{Complex geometry of moment-angle manifolds}

\author{Taras Panov}
\address{Department of Mathematics and Mechanics, Moscow
State University, Leninskie Gory, 119991 Moscow, Russia,
\newline\indent Institute for Theoretical and Experimental Physics,
Moscow, Russia,
\newline\indent Institute for Information Transmission Problems,
Russian Academy of Sciences}
\email{tpanov@mech.math.msu.su}

\author{Yury Ustinovskiy}
\address{Steklov Institute of Mathematics, 8 Gubkina str., 119991 Moscow, Russia
\newline\indent Department of Mathematics, Princeton University,
Princeton, USA
}
\email{yura.ust@gmail.com}

\author{Misha Verbitsky}
\address{Laboratory of Algebraic Geometry, Department of Mathematics,
National Research University ``Higher School of Economics'', 7
Vavilova str., Moscow, Russia
\newline\indent Universit\'e Libre de Bruxelles, CP 218,
Bd du Triomphe, 1050 Brussels, Belgium}

\email{verbit@mccme.ru}

\thanks{The first and the second
authors were supported by the Russian Science Foundation, RSCF
grant 14-11-00414 at the Steklov Institute of Mathematics. The
third author was partially supported by the RSCF grant 14-21-00053
within the AG Laboratory NRU-HSE}




\begin{abstract}
Moment-angle manifolds provide a wide class of examples of
non-K\"ahler compact complex manifolds. A complex moment-angle
manifold $\mathcal Z$ is constructed via certain combinatorial
data, called a complete simplicial fan. In the case of rational
fans, the manifold $\mathcal Z$ is the total space of a
holomorphic bundle over a toric variety with fibres compact
complex tori. In general, a complex moment-angle manifold
$\mathcal Z$ is equipped with a canonical holomorphic foliation
${\mathcal F}$ which is equivariant with respect to the
$({\C}^\times)^m$-action. Examples of moment-angle manifolds
include Hopf manifolds of Vaisman type, Calabi--Eckmann manifolds,
and their deformations.

We construct transversely K\"ahler metrics on moment-angle
manifolds, under some restriction on the combinatorial data. We
prove that any K\"ahler submanifold (or, more generally, a Fujiki
class $\mathcal C$ subvariety) in such a moment-angle manifold is
contained in a leaf of the foliation~${\mathcal F}$. For a generic
moment-angle manifold~$\mathcal Z$ in its combinatorial class, we
prove that all subvarieties are moment-angle manifolds of smaller
dimension and there are only finitely many of them. This implies,
in particular, that the algebraic dimension of~$\mathcal Z$ is
zero.
\end{abstract}

\maketitle

\tableofcontents

\section{Introduction}

\subsection{Non-K\"ahler geometry}\label{nKg}
The class of non-K\"ahler complex manifolds is by several orders
of magnitude bigger than that of K\"ahler ones. For instance, any
finitely presented group can be realized as a fundamental group of
a compact complex 3-manifold.\footnote{By a result of C. Taubes.
See~\cite{_Panov_Petrunin_} for a simpler proof and references to
earlier works.} However, fundamental groups of compact K\"ahler
manifolds are restricted by many different constraints, and
apparently belong to a minuscule subclass of the class of all
finitely presented groups~\cite{_Amoros_etc_}.

Despite the abundance of non-K\"ahler complex manifolds, there are
only few explicit constructions of them. Historically the first
example was the \emph{Hopf surface}, obtained as the quotient of
${\C}^2\setminus\{\bf 0\}$ by the action of $\Z$ generated by the
transformation $\mb z\mapsto 2\mb z$. This construction can be
generalised by considering  quotients of ${\C}^n\setminus\{\bf
0\}$, $n\ge2$, by the action of $\Z$ generated by a linear
operator with all eigenvalues $\alpha_i$ satisfying
$|\alpha_i|>1$. The resulting quotients are called
\emph{(generalised) Hopf manifolds}. Every Hopf manifold is
diffeomorphic to $S^1 \times S^{2n-1}$. It is clearly
non-K\"ahler, as its second cohomology group is zero. A particular
class of generalised Hopf manifolds also features in the work of
Vaisman~\cite{_Vaisman:non-Kahler_}, where he proved that any
compact locally conformally K\"ahler manifold is non-K\"ahler. We
discuss these in more detail in
Subsection~\ref{_Vaisman_Subsection_}.


Another generalisation of the Hopf construction is due to Calabi
and Eckmann. A Calabi--Eckmann manifold is diffeomorphic to a
product of even number of odd-dimensional spheres: $E= S^{2n_1+1}
\times S^{2n_2+1} \times \cdots \times S^{2n_{2k}+1}$. It is the
total space of a principal $T^{2k}$-bundle over ${\C} P^{n_1}
\times {\C} P^{n_2} \times \cdots \times  {\C} P^{n_{2k}}$.

Any principal $T^l$-bundle $E$ over  a base $B$ is determined
topologically by the Chern classes $\nu_i\in H^2(B, \Z)$, \
$i=1,\ldots,l$. These Chern classes can be obtained as the
cohomology classes of the curvature components of a connection
on~$E$. Now, suppose that $B$ is complex and these curvature
components are forms of type (1,1). In this case, for each complex
structure on the fibre $T^l$ with even~$l$, the total space
becomes a complex manifold, which is easy to see if one expresses
the commutators of vector fields tangent to $E$ through the
connection.

For the Calabi--Eckmann fibration,
\[
 E= S^{2n_1+1} \times S^{2n_2+1} \times \cdots \times S^{2n_{2k}+1}
\stackrel{T^{2k}} \longrightarrow B={\C} P^{n_1} \times {\C}
P^{n_2} \times \cdots \times  {\C} P^{n_{2k}},
\]
the natural connection on fibres has the respective Fubini--Study
forms as its curvatures. This makes $E$ into a complex manifold,
holomorphically fibred over $B$.

Applying this to a $T^2$-fibration $S^1 \times S^{2l-1}\to {\C}
P^{l-1}$, we obtain the classical Hopf manifold, which is also the
quotient of ${\C}^n\backslash\{\bf0\}$ by an action of $\Z$ given
by $\mb z\to \lambda\mb z$, $\lambda\in {\C}$, $|\lambda|>1$. By
deforming this action we obtain a manifold which does not admit
elliptic fibrations. However, there is still a holomorphic
\emph{foliation}, and its differential-geometric properties remain
qualitatively the same (Subsection \ref{_Vaisman_Subsection_}).

In the present paper, we study complex manifolds which arise in a
similar fashion.

Consider a projective (or Moishezon) toric manifold $B$, that is,
a complex manifold equipped with an action of $({\C}^\times)^n$
which has an open orbit. Fixing an even number of cohomology
classes $\eta_1, \ldots, \eta_{2k}\in H^2(B;\Z)$, one obtains a
$T^{2k}$-fibration $\mathcal E\to B$ with the Chern
classes~$\eta_i$. Since the classes $\eta_i$ are represented by
$(1,1)$-forms, the corresponding connection makes $\mathcal E$
into a complex manifold, similar to the Calabi--Eckmann example.
Complex moment-angle manifolds~$\mathcal Z$, which are the main
objects of study in this paper, may be thought of as deformations
of~$\mathcal E$, using a combinatorial procedure explained in
detail in Section~\ref{_complex_ana_stru_Section_}.

The topology of $\mathcal Z$ depends only on a simplicial
complex~$\sK$ on $m$ elements, which is reflected in the notation
$\mathcal Z=\zk$, while the complex structure comes from a
realisation of $\zk$ as the quotient of an open subset $U(\sK)$ in
$\C^m$ by a holomorphic action of a group $C$ isomorphic
to~$\C^\ell$. In order that the action of $C$ be free and proper
one needs to assume that $\sK$ is the underlying complex of a
complete simplicial fan~$\Sigma$.


\subsection{Positive (1,1)-forms and currents on non-K\"ahler
manifolds}\label{sscurrents}

Let $M$ be a compact complex manifold. A differential $(1,1)$-form
$\beta$ on $M$ is called \emph{positive} if it is real and
$\beta(V,JV)\ge0$ for any real tangent vector~$V\in TM$, where $J$
is the operator of the complex structure. Equivalently, a
$(1,1)$-form $\beta=i\sum_{j,k}f_{jk}dz_j\wedge d \overline{z}_k$
is positive if $\sum_{j,k}f_{jk}\xi_j\overline{\xi}_k$ is a
Hermitian positive semidefinite form.


A \emph{$(1,1)$-current} is a linear functional $\Theta$ on the
space $\Lambda^{1,1}_c(M)$ of forms with compact support which is
continuous in one of the $C^k$-topologies defined by the
$C^k$-norm $|\phi|_{C^k}=\sup_{m\in M} \sum_{i=0}^k |\nabla^i
\phi|$. A $(1,1)$-current $\Theta$ is called \emph{positive} if
$\langle \Theta, \beta\rangle \ge 0$ for any positive
$(1,1)$-form~$\beta$. See~\cite{_Harvey_Lawson:currents_}
and~\cite{dema12} for more details on the notions of positive
forms and currents.



The main result of Harvey and Lawson's
work~\cite{_Harvey_Lawson:currents_} states that $M$ does not
admit a K\"ahler metric if and only if $M$ has a nonzero positive
$(1,1)$-current $\Theta$ which is the $(1,1)$-component of an
exact current.
Any complex surface (2-dimensional manifold) $M$ admits a closed
positive (1,1)-current \cite{_Buchdahl:surfaces_,_Lamari_}, and
sometimes a closed positive $(1,1)$-form. The maximal rank of such
form is called the \emph{K\"ahler rank} of~$M$. In
\cite{_Brunella:Inoue_} and \cite{_Chiose_Toma:surfaces_}, complex
surfaces were classified according to their K\"ahler rank; it was
shown that the K\"ahler rank is equal to 1 for all elliptic, Hopf
and Inoue minimal surfaces, and 0 for the rest of class VII
surfaces.

%


%
%

One of the most striking advances in complex geometry of the 2000s
was the discovery of \emph{Oeljeklaus--Toma
manifolds}~\cite{_Oeljeklaus_Toma_}. These are complex
solvmanifolds associated with all number fields admitting complex
and real embeddings. When the number field has the lowest possible
degree (that is, cubic), the corresponding Oeljeklaus--Toma
manifold is an \emph{Inoue surface} of type~$S^0$, hence
Oeljeklaus--Toma manifolds can be considered as generalisations of
Inoue surfaces. In \cite{_OV:Oeljeklaus-Toma_}, a positive exact
$(1,1)$-form $\omega_0$ was constructed on any Oeljeklaus--Toma
manifold. Using this form, it was proved that an Oeljeklaus--Toma
manifold has no curves (see~\cite{_Sima:curves_}), and that an
Oeljeklaus--Toma manifold has no positive dimensional subvarieties
when the number field has precisely~1 complex embedding, up to
complex conjugation (see~\cite{_OV:Oeljeklaus-Toma_}). The
argument, in both cases, is similar to that used in the present
paper: a complex subvariety $Z\subset M$ cannot be transverse to
the zero foliation ${\mathcal F}$ of $\omega_0$, as otherwise
$\int_Z \omega_0^{\dim_{\C} Z}>0$, which is impossible by the
exactness of~$\omega_0$. The rank of ${\mathcal F}$ is equal to
the number of real embeddings. For the case considered
in~\cite{_OV:Oeljeklaus-Toma_} it is~1, hence any subvariety
$Z\subset M$ must contain a leaf of~${\mathcal F}$. The argument
of~\cite{_OV:Oeljeklaus-Toma_} is finished by applying the
``strong approximation theorem'' (a result of number theory) to
prove that the closure of a leaf of ${\mathcal F}$ cannot be
contained in a subvariety of dimension less than $\dim M$.

\subsection{Hopf manifolds and Vaisman manifolds}
\label{_Vaisman_Subsection_}

Positive exact (1,1)-forms are also very useful in the study of
Vaisman manifolds. The latter form a special class of
\emph{locally conformally K\"ahler manifold} (\emph{LCK manifolds}
for short). An LCK manifold is a manifold $M$ whose universal
covering $\widetilde M$ has a K\"ahler metric with the monodromy
of the covering acting on $\widetilde M$ by non-isometric
homotheties. An LCK manifold is called \emph{Vaisman} if its
K\"ahler covering $\widetilde M$ admits a monodromy-equivariant
action of ${\C}^\times$ by non-isometric holomorphic homotheties.

The original definition of Vaisman manifolds
(see~\cite{_Dragomir_Ornea_} for more details, history and
references) builds upon an explicit differential-geometric
construction, and their description given above is a result of
Kamishima--Ornea~\cite{_Kamishima_Ornea_}. The existence of
positive exact $(1,1)$-forms of Vaisman manifolds is shown
in~\cite{_Verbitsky_vanishing_}.

Vaisman called his manifolds ``generalised Hopf manifolds'';
however, his construction does not give all generalised Hopf
manifolds in the sense of Subsection~\ref{nKg}. Indeed, it is
proved in~\cite[Theorem 3.6]{_OV:LCK_pot_} that a Hopf manifold
${\C}^n \setminus\{\mathbf0\}/\{\mb z\sim A\mb z\}$ corresponding
to a linear operator $A$ with all eigenvalues~$>1$ is Vaisman if
and only if $A$ is diagonalisable (the ``if'' part follows from
\cite{_Gauduchon_Ornea_}).

%

Hopf manifolds of Vaisman type are particular cases of
moment-angle manifolds (see Example~\ref{hopf}). Therefore, our
description of complex submanifolds in moment-angle manifolds
generalises the results obtained for Hopf manifolds by Kato
in~\cite{_Kato:submanifolds_}.

\subsection{Moment-angle manifolds and their geometry}
A moment-angle complex $\zk$ is a cellular subcomplex in the unit
polydisc $\D^m\subset{\C}^m$ composed of products of discs and
circles~\cite{bu-pa00}. These products are parametrised by a
finite simplicial complex~$\sK$ on a set of $m$ elements. Each
moment-angle complex $\zk$ carries a natural action of the
$m$-torus~$T^m$. Moment-angle complexes have many interesting
homotopical and geometric properties, and have been studied
intensively over the last ten years (see~\cite{bu-pa00,b-b-c-g10,
gi-lo13,gr-th13,pano13}). When $\sK$ is the boundary of a
simplicial polytope or, more generally, a triangulated sphere, the
moment-angle complex $\zk$ is a manifold, referred to as a
\emph{moment-angle manifold}.

Moment-angle manifolds arising from polytopes can be realised by
nondegenerate intersections of Hermitian quadrics in~${\C}^m$.
These intersections of quadratic surfaces feature in holomorphic
dynamics as transverse spaces for complex foliations
(see~\cite{bo-me06} for a historic account of these developments),
and in symplectic toric topology as level sets for the moment maps
of Hamiltonian torus actions on~${\C}^m$ (see~\cite{audi91}). As
was discovered by Bosio and Meersseman in~\cite{bo-me06},
moment-angle manifolds arising from polytopes admit complex
structures as~\emph{LVM-manifolds}. This led to establishing a
fruitful link between the theory of moment-angle manifolds and
complex geometry.

An alternative construction of complex structures on moment-angle
manifolds was given in~\cite{pa-us12}; it works not only for
manifolds~$\zk$ arising from polytopes, but also when $\sK$ is the
underlying complex of a complete simplicial fan~$\Sigma$. The
complex structure on $\zk$ comes from its realisation as the
quotient of an open subset $U(\sK)\subset\C^m$ (the complement to
a set of coordinate subspaces determined by~$\sK$) by an action of
$\C^\ell$ embedded as a holomorphic subgroup in the
$\C^\times$-torus $(\C^\times)^m$. As in toric geometry, the
simplicial fan condition guarantees that the quotient is
Hausdorff~\cite{pano13}.

The complex moment-angle manifold $\zk$ is non-K\"ahler, except
for the case when $\sK$ is an empty simplicial complex and $\zk$
is a compact complex torus. When the simplicial fan $\Sigma$ is
rational, the manifold $\zk$ is the total space of a holomorphic
principal (Seifert) bundle with fibre a complex torus and base the
toric variety~$V_\Sigma$. (The polytopal case was studied by
Meersseman and Verjovsky in~\cite{me-ve04}; the situation is
similar in general, although the base need not to be projective.)
We therefore obtain a generalisation of the classical families of
Hopf and Calabi--Eckmann manifolds (which are the total spaces of
fibrations with fibre an elliptic curve over $\C P^n$ and $\C
P^k\times\C P^l$, respectively). In this case, although $\zk$ is
not K\"ahler, its complex geometry is quite rich, and meromorphic
functions, divisors and Dolbeault cohomology can be described
explicitly~\cite{me-ve04,pa-us12}.

The situation is completely different when the data defining the
complex structure on~$\zk$ is generic (in particular, the
fan~$\Sigma$ is not rational). In this case, the moment-angle
manifold $\zk$ does not admit global meromorphic functions and
divisors, and behaves similarly to the complex torus obtained as
the quotient of $\C^m$ by a generic lattice of full rank.
Furthermore, $\zk$ has only finitely many complex submanifolds of
a very special type: they are all ``coordinate'', i.\,e. obtained
as quotients of coordinates subspaces in~$\C^m$.

Our study of the complex geometry of moment-angle manifolds $\zk$
builds upon two basic constructions: a holomorphic
foliation~$\mathcal F$ on $\zk$ (which may be thought of as the
result of deformation of the holomorphic fibration $\mathcal E\to
B$ described in Subsection~\ref{nKg}), and a positive closed
$(1,1)$-form $\omega$ whose zero spaces define the foliation $\cF$
over a dense open subset of~$\zk$. Forms with these properties are
called \emph{transverse K\"ahler} for~$\cF$. The precise result is
as follows:

\begingroup
\def\thetheorem{\ref{form2}}
\begin{theorem}
Assume that $\Sigma$ is a weakly normal fan defined by $\{\mathcal
K;\mb a_1,\ldots,\mb a_m\}$. Then there exists an exact
$(1,1)$-form~$\omega_\cF$ on $\zk=U(\sK)/C$ which is transverse
K\"ahler for the foliation~$\cF$ on the dense open subset
$(\C^\times)^m/C\subset U(\sK)/C$.
\end{theorem}
\addtocounter{theorem}{-1}
\endgroup

Transverse K\"ahler forms are often used to study the complex
geometry (subvarieties, stable bundles) of non-K\"ahler
manifolds~\cite{_OV:Oeljeklaus-Toma_,_Verbitsky:Sta_Elli_,
_Verbitsky:stable-Hopf_,_Verbitsky:toric_}. In our case, Fujiki
class~$\mathcal{C}$ subvarieties (in particular, K\"ahler
submanifolds) of moment-angle manifolds are described as follows:

\begingroup
\def\thetheorem{\ref{kahlsub}}
\begin{theorem}
Under the assumptions of Theorem~{\rm\ref{form2}}, any Fujiki
class $\mathcal C$ subvariety of $\zk$ is  contained in a leaf of
the $\ell$-dimensional foliation~$\cF$.
\end{theorem}
\addtocounter{theorem}{-1}
\endgroup

An important consequence is that under generic assumption on the
complex structure (when a generic leaf of~$\cF$ is biholomorphic
to~$\C^\ell$) there are actually \emph{no} Fujiki class $\mathcal
C$ subvarieties of $\zk$ through a generic point
(Corollary~\ref{leavescor}).

We proceed by studying general subvarieties of $\zk$ (not
necessarily of Fujiki class~$\mathcal C$). In the case of
codimension one (divisors), generically there are only finitely
many of them (Theorem~\ref{nodiv}). The proof does not use a
transverse transverse K\"ahler form (and therefore no geometric
restrictions on the fan need to be imposed); instead we reduce the
statement to the case of Hopf manifolds, where it is proved by
analysing one-dimensional foliations. A corollary is that a
generic $\zk$ does not have non-constant meromorphic functions, so
its algebraic dimension is zero.

The foliation ${\mathcal F}$ comes from a group action
(Construction \ref{Ffoli}), so its tangent bundle $T{\mathcal
F}\subset T\zk$ is trivial. Given an irreducible subvariety $Y
\subset\zk$, consider the number $k=\dim T_yY\cap T_y{\mathcal F}$
for $y$ a generic point in~$Y$. Since $T{\mathcal F}$ is a trivial
bundle, there is a rational map $\psi$ from $Y$ to the
corresponding Grassmann manifold of $k$-planes in $T_y{\mathcal
F}$. Using the absence of Fujiki class $\mathcal{C}$ submanifolds
in $\zk$, we deduce by induction on $\dim Y$ that $\psi$ is
constant (see the details in Subsection~\ref{gensubm}).
Generically this cannot happen when $Y$ contains a point from an
open subset $(\C^\times)^m/C\subset\zk$, whose complement consists
of ``coordinate'' submanifolds. We therefore prove

\begingroup
\def\thetheorem{\ref{final}}
\begin{theorem}
Let $\zk$ be a complex moment-angle manifold, with linear maps
$A\colon\R^m\to N_\R$ and $\varPsi\colon\C^\ell\to\C^m$ described
in Construction~{\rm\ref{psi}}. Assume that
\begin{itemize}
\item[(a)]no rational linear function on
$\R^m$ vanishes on $\Ker A$;
\item[(b)] $\Ker A$ does not contain rational vectors of~$\R^m$;
\item[(c)]the map $\varPsi$ satisfies the generic condition of
Lemma~{\rm\ref{generic_cond}};
\item[(d)] the fan $\Sigma$ is weakly normal.
\end{itemize}
Then any irreducible analytic subset $Y\subsetneq\zk$ of
positive dimension is contained in a coordinate submanifold $\mathcal Z_{\mathcal K_J}\subsetneq \zk$.
\end{theorem}
\addtocounter{theorem}{-1}
\endgroup

Under further assumption on the complex structure of $\zk$ we
prove that all irreducible analytic subsets are actually
coordinate submanifolds (Corollary~\ref{final2}).

A similar idea was used by Dumitrescu to classify holomorphic
geometries on non-K\"ahler manifolds~\cite{_Dumitrescu:holo_stru_,
_Dumitrescu:holo_met_}, and also much earlier by Bogomolov in his
work \cite{_Bogomolov:tensors_} on holomorphic tensors and
stability.

It has been recently shown by Ishida~\cite{isid} that complex
moment-angle manifolds $\zk$ have the following universal
property: any compact complex manifold with a holomorphic maximal
action of a torus is biholomorphic to the quotient of a
moment-angle manifold by a freely acting closed subgroup of the
torus. (An effective action of $T^k$ on an $m$-dimensional
manifold $M$ is \emph{maximal} if there exists a point $x\in M$
whose stabiliser has dimension $m-k$; the two extreme cases are
the free action of a torus on itself and the half-dimensional
torus action on a toric manifold.)

\smallskip

In Section~2 we review the definition of moment-angle
manifolds~$\zk$ and their realisation by nondegenerate
intersections of Hermitian quadrics in the polytopal case. In
Section~3 we describe complex structures on moment-angle
manifolds. Section~4 contains the main results. We start with
defining a holomorphic foliation~$\mathcal F$ which replaces the
holomorphic fibration $\zk\to V_\Sigma$ over a toric variety.
Theorem~\ref{form2} is the main technical result, which may be of
independent interest. It provides an exact $(1,1)$-form
$\omega_{\mathcal F}$ which is transverse K\"ahler for the
foliation~$\cF$ on a dense open subset. In Theorem~\ref{kahlsub}
we show that, under a generic assumption on the data defining the
complex structure, any K\"ahler submanifold (more generally, a
Fujiki class~$\mathcal C$ subvariety) of $\zk$ is contained in a
leaf of~$\cF$. In Theorem~\ref{nodiv} we show that a generic
moment-angle manifold~$\zk$ has only coordinate divisors. This
result is independent of the existence of a transverse K\"ahler
form. Consequently, a generic~$\zk$ does not admit non-constant
global meromorphic functions (Corollary~\ref{nomerfun}). In
Theorem~\ref{final} and Corollary~\ref{final2} we prove that any
irreducible analytic subset of positive dimension in a generic
moment-angle manifold is a coordinate submanifold.

\medskip

\textbf{Acknowledgements.} We thank the referees for their most
helpful comments and suggestions.

\section{Basic constructions}
Let $\mathcal K$ be an abstract \emph{simplicial complex} on the
set $[m]=\{1,\ldots,m\}$, that is, a collection of subsets
$I=\{i_1,\ldots,i_k\}\subset[m]$ closed under inclusion. We refer
to $I\in\sK$ as (abstract) \emph{simplices} and always assume that
$\varnothing\in\sK$. We do not assume that $\sK$ contains all
singletons $\{i\}\subset[m]$, and refer to $\{i\}\notin\sK$ as a
\emph{ghost vertex}.


Consider the closed unit polydisc in $\C^m$,
\[
  \D^m=\bigl\{ (z_1,\ldots,z_m)\in\C^m\colon |z_i|\le1\bigr\}.
\]
Given $I\subset[m]$, define
\[
  B_I=\bigl\{(z_1,\ldots,z_m)\in
  \D^m\colon |z_j|=1\text{ for }j\notin I\bigl\},
\]
Following~\cite{bu-pa00}, define the \emph{moment-angle complex}
$\zk$ as
\begin{equation}\label{defzk}
  \zk=\bigcup_{I\in\sK}B_I
  =\bigl\{(z_1,\ldots,z_m)\in\D^m\colon\{i\colon|z_i|<1\}\in\sK\bigr\}.
\end{equation}

This is a particular case of the following general construction.

\begin{construction}[polyhedral product]
Let $X$ be a topological space, and $A$ a subspace of~$X$. Given
$I\subset[m]$, set
\begin{equation}\label{xwi}
  (X,A)^I=\bigl\{(x_1,\ldots,x_m)\in X^m\colon x_j\in A\text{ for
  }j\notin I\bigr\}\cong\prod_{i\in I}X\times\prod_{i\notin I}A,
\end{equation}
and define the \emph{polyhedral product} of $(X,A)$ as
\[
  (X,A)^\sK=\bigcup_{I\in\sK}(X,A)^I\subset X^m.
\]
Obviously, if $\sK$ has $k$ ghost vertices, then
$(X,A)^\sK\cong(X,A)^{\mathcal K'}\times A^k$, where $\mathcal K'$
does not have ghost vertices.

We have $\zk=\zk(\D,\T)$, where $\T$ is the unit circle. Another
important particular case is the complement of a \emph{complex
coordinate subspace arrangement}:
\begin{equation}\label{uk}
  U(\sK)=(\C,\C^\times)^\sK=\C^m\big\backslash\bigcup_{\{i_1,\ldots,i_k\}\notin\sK}
  \{\mb z\in\C^m\colon z_{i_1}=\cdots=z_{i_k}=0\},
\end{equation}
where $\C^\times=\C\setminus\{0\}$. We obviously have $\zk\subset
U(\sK)$. Moreover, $\zk$ is a deformation retract of $U(\sK)$ for
every $\sK$~\cite[Theorem~4.7.5]{bu-pa15}.
\end{construction}

According to~\cite[Theorem~4.1.4]{bu-pa15}, $\zk$ is a (closed)
topological manifold of dimension $m+n$ whenever $\sK$ defines a
simplicial subdivision of a sphere $S^{n-1}$; in this case we
refer to $\zk$ as a \emph{moment-angle manifold}.

An important geometric class of simplicial subdivisions of spheres
is provided by \emph{starshaped spheres}, or underlying complexes
of complete simplicial fans.

Let $N_\R\cong\R^n$ be an $n$-dimensional space.  A
\emph{polyhedral cone} $\sigma$ is the set of nonnegative linear
combinations of a finite collection of vectors $\mb a_1,\ldots,\mb
a_k$ in $N_\R$:
\[
  \sigma=\{\mu_1\mb a_1+\cdots+\mu_k\mb
  a_k\colon\mu_i\ge0\}.
\]
A polyhedral cone is \emph{strongly convex} if it contains no line. A strongly convex polyhedral cone is \emph{simplicial} if
it is generated by a subset of a basis of~$N_\R$.

A \emph{fan} is a finite collection
$\Sigma=\{\sigma_1,\ldots,\sigma_s\}$ of strongly convex
polyhedral cones in $N_\R$ such that every face of a cone in
$\Sigma$ belongs to $\Sigma$ and the intersection of any two cones
in $\Sigma$ is a face of each. A fan $\Sigma$ is \emph{simplicial}
if every cone in $\Sigma$ is simplicial. A fan
$\Sigma=\{\sigma_1,\ldots,\sigma_s\}$ is called \emph{complete} if
$\sigma_1\cup\cdots\cup\sigma_s=N_\R$.

{\samepage A simplicial fan $\Sigma$ in $N_\R$ is therefore
determined by two pieces of data:
\begin{itemize}
\item[--] a simplicial complex $\sK$ on $[m]$;
\item[--] a configuration of vectors $\mb a_1,\ldots,\mb a_m$ in
$N_\R$ such that the subset $\{\mb a_i\colon i\in I\}$ is linearly
independent for any simplex $I\in\sK$.
\end{itemize}}
Then for each $I\in\sK$ we can define the simplicial cone
$\sigma_I$ spanned by $\mb a_i$ with $i\in I$. The ``bunch of
cones'' $\{\sigma_I\colon I\in\sK\}$ patches into a fan $\Sigma$
whenever any two cones $\sigma_I$ and $\sigma_J$ intersect in a
common face (which has to be $\sigma_{I\cap J}$). Equivalently,
the relative interiors of cones $\sigma_I$ are pairwise
non-intersecting. Under this condition, we say that the data
$\{\sK;\mb a_1,\ldots,\mb a_m\}$ \emph{define a fan}~$\Sigma$, and
$\sK$ is the \emph{underlying simplicial complex} of the
fan~$\Sigma$. Note that a simplicial fan $\Sigma$ is complete if
and only if its underlying simplicial complex $\sK$ defines a
simplicial subdivision of~$S^{n-1}$; this simplicial subdivision
is obtained by intersecting $\Sigma$ with a unit sphere in~$N_\R$.

Here is an important point in which our approach to fans differs
from the standard one adopted in toric geometry: since we allow
ghost vertices in~$\sK$, we do not require that each vector $\mb
a_i$ spans a one-dimensional cone of~$\Sigma$. The vector $\mb
a_i$ corresponding to a ghost vertex $\{i\}\in[m]$ may be zero.


In the case when $\{\sK;\mb a_1,\ldots,\mb a_m\}$ define a
complete fan~$\Sigma$, the topological manifold $\zk$ can be
smoothed by means of the following procedure~\cite{pa-us12}:

\begin{construction}
Let $\mb a_1,\ldots,\mb a_m$ be a set of vectors. We consider the
linear map
\begin{equation}\label{lambdar}
  A\colon\R^m\to N_\R,\quad\mb e_i\mapsto\mb a_i,
\end{equation}
where $\mb e_1,\ldots,\mb e_m$ is the standard basis of~$\R^m$.
Let
\[
  \R^m_>=\{(y_1,\ldots,y_m)\in\R^m\colon y_i>0\}
\]
be the multiplicative group of $m$-tuples of positive real
numbers, and define
\begin{equation}\label{rsigma}
\begin{aligned}
  R&=\exp(\Ker A)=\bigl\{\bigl(e^{y_1},\ldots,e^{y_m}\bigr)
  \colon(y_1,\ldots,y_m)\in\Ker A\bigr\}\\
  &=\bigl\{(t_1,\ldots,t_m)\in\R^m_>\colon
  \prod_{i=1}^mt_i^{\langle\mb a_i,\mb u\rangle}=1
  \text{ for all }\mb u\in N^*_\R\bigr\}.
\end{aligned}
\end{equation}

We let $\R^m_>$ act on $\C^m$ by coordinatewise multiplications.
The subspace $U(\sK)$~\eqref{uk} is $\R^m_>$-invariant for
any~$\sK$, and there is an action of $R\subset\R^m_>$ on $U(\sK)$
by restriction.
\end{construction}

\begin{theorem}[{\cite[Theorem~2.2]{pa-us12}}]\label{zksmooth}
Assume that data $\{\sK;\mb a_1,\ldots,\mb a_m\}$ define a
complete fan~$\Sigma$ in~$N_\R\cong\R^n$. Then
\begin{itemize}
\item[(a)]
the group $R\cong\R^{m-n}$ acts on $U(\sK)$ freely and properly,
so the quotient $U(\sK)/R$ is a smooth $(m+n)$-dimensional
manifold;

\item[(b)] $U(\sK)/R$ is homeomorphic to~$\zk$.
\end{itemize}
\end{theorem}

\begin{remark}
The group $R$ depends on $\mb a_1,\ldots,\mb a_m$. However, we
expect that the smooth structure on $\zk$ as the quotient
$U(\sK)/R$ does not depend on this choice, i.e. smooth manifolds
corresponding to different $R$ are diffeomorphic.
\end{remark}

An important class of complete fans arises from convex polytopes:

\begin{construction}[Normal fan]
Let $P$ be a convex polytope in the dual space $N_\R^*\cong\R^n$.
It can be written as a bounded intersection of $m$ halfspaces:
\begin{equation}
\label{ptope}
  P=\bigl\{\mb u\in N_\R^*\colon\langle\mb a_i,\mb u\rangle+b_i\ge0\quad\text{for }
  i=1,\ldots,m\bigr\},
\end{equation}
where $\mb a_i\in N_\R$ and $b_i\in\R$.

We assume that the intersection $P\cap\{\langle\mb a_i,\mb
u\rangle+b_i=0\}$ with the $i$th hyperplane is either empty or
$(n-1)$-dimensional. In the latter case, $F_i=P\cap\{\langle\mb
a_i,\mb u\rangle+b_i=0\}$ is a \emph{facet} of~$P$. We allow
$P\cap\{\langle\mb a_i,\mb u\rangle+b_i=0\}$ to be empty for
technical reasons explained below, in this case the $i$th
inequality $\langle\mb a_i,\mb u\rangle+b_i\ge0$ is
\emph{redundant}. Our assumption also implies that $P$ has full
dimension~$n$.

A \emph{face} of $P$ is a nonempty intersection of facets. Given a
face $Q\subset P$, define
\[
  \sigma_Q=\{\mb x\in N_\R\colon
  \langle\mb x,\mb u'\rangle\le\langle\mb x,\mb u\rangle
  \text{ for all $\mb u'\in Q$ and $\mb u\in P$}\}.
\]
Each $\sigma_Q$ is a strongly convex cone, and the collection
$\{\sigma_Q\colon Q\text{ is a face of }P\}\cup\{\bf0\}$ is a
complete fan in~$N_\R$, called the \emph{normal fan} of~$P$ and
denoted by~$\Sigma_P$.

An $n$-dimensional polytope $P$ is called \emph{simple} if exactly
$n$ facets meet at each vertex of~$P$. In a simple polytope, any
face $Q$ of codimension $k$ is uniquely written as an intersection
of $k$ facets $Q=F_{i_1}\cap\cdots\cap F_{i_k}$. Furthermore, the
cone $\sigma_Q$ is simplicial and is generated by the $k$ normal
vectors $\mb a_{i_1},\ldots,\mb a_{i_k}$ to the corresponding
facets.

Define the \emph{nerve complex} of a polytope $P$ with $m$ facets
by
\begin{equation}\label{simpo}
  \sK_P=\{I\subset[m]\colon \bigcap_{i\in I} F_i\ne\varnothing\}.
\end{equation}
If $P$ is simple then $\Sigma_P$ is the simplicial fan defined by
the data $\{\sK_P;\mb a_1,\ldots,\mb a_m\}$.
\end{construction}

\begin{remark}
Not every complete fan is a normal fan of a polytope
(see~\cite[\S3.4]{fult93} for an example). Even the weaker form of
this statement fails: there are complete simplicial fans $\Sigma$
whose underlying complexes $\sK$ are not combinatorially
equivalent to the boundary of a convex polytope. The
\emph{Barnette sphere} (a non-polytopal triangulation of $S^3$
with 8 vertices) provides a
counterexample~\cite[\S{}III.5]{ewal96}.
\end{remark}

In the case when $\sK=\sK_P$ for a simple polytope~$P$, there is
an alternative way to give $\zk$ a smooth structure, by writing it
as a nondegenerate intersection of Hermitian quadrics:

\begin{construction}\label{dist}
Let $P$ be a simple convex polytope given by~\eqref{ptope}. We
consider the affine map
\[
  i_P\colon N^*_\R\to\R^m,\quad\mb u\mapsto
  \bigl(\langle\mb a_1,\mb u\rangle+b_1,\ldots,
  \langle \mb a_m,\mb u\rangle+b_m\bigr),
\]
or $i_P(\mb u)=A^*\mb u+\mb b$, where $A^*$ is the dual map
of~\eqref{lambdar} and $\mb b\in\R^m$ is the vector with
coordinates~$b_i$. The image $i_P(N_\R^*)$ is an affine $n$-plane
in $\R^m$, which can be written by $m-n$ linear equations:
\begin{equation}\label{iabrn}
\begin{aligned}
  i_P(N_\R^*)&=\{\mb y\in\R^m\colon\mb y=A^*\mb u+\mb b\quad
  \text{for some }\mb u\in N_\R^*\}\\
  &=\{\mb y\in\R^m\colon\varGamma\mb y=\varGamma\mb b\},
\end{aligned}
\end{equation}
where $\varGamma=(\gamma_{jk})$ is an $(m-n)\times m$-matrix whose
rows form a basis of linear relations between the vectors~${\mb
a}_i$. That is, $\varGamma$ is of full rank and satisfies the
identity $\varGamma A^*=0$.

The map $i_P$ embeds $P\subset\R^n$ into the positive orthant
$\R^m_\ge$. We define the space $\mathcal Z_P$ from the
commutative diagram
\begin{equation}\label{cdiz}
\begin{CD}
  \mathcal Z_P @>i_{\mathcal Z}>>\C^m\\
  @VVV\hspace{-0.2em} @VV\mu V @.\\
  P @>i_P>> \R^m_\ge
\end{CD}
\end{equation}
where $\mu(z_1,\ldots,z_m)=(|z_1|^2,\ldots,|z_m|^2)$. The torus
$\T^m$ acts on $\mathcal Z_P$ with quotient $P$, and $i_{\mathcal
Z}$ is a $\T^m$-equivariant embedding.

By replacing $y_k$ by $|z_k|^2$ in the equations~\eqref{iabrn}
defining the affine plane $i_P(N_\R^*)$ we obtain that $\mathcal
Z_P$ embeds into $\C^m$ as the set of common zeros of $m-n$ real
quadratic equations (\emph{Hermitian quadrics}):
\begin{equation}\label{zpqua}
  i_{\mathcal Z}(\mathcal Z_P)=\Bigl\{\mb z\in\C^m\colon\sum_{k=1}^m\gamma_{jk}|z_k|^2=
  \sum_{k=1}^m\gamma_{jk}b_k\quad\text{for }1\le j\le m-n\Bigr\}.
\end{equation}
We identify $\zp$ with its image $i_{\mathcal Z}(\zp)$ and
consider $\zp$ as a subset of~$\C^m$.
\end{construction}

\begin{theorem}[{\cite{bu-pa00,pa-us12}}]\label{Rzkzp}
Let $P$ be a simple polytope given by~\eqref{ptope} and let
$\sK=\sK_P$ be its nerve simplicial complex~\eqref{simpo}. Then
\begin{itemize}
\item[(a)] $\zp\subset U(\sK)$;
\item[(b)] $\zp$ is a smooth submanifold in $\C^m$ of dimension~$m+n$;
\item[(c)] there is a $\T^m$-equivariant homeomorphism
$\zk\cong\zp$.
\end{itemize}
\end{theorem}
\begin{proof}[Sketch of proof]
(a) By~\eqref{cdiz}, if a point $\mb z=(z_1,\ldots,z_m)\in\zp$ has
its coordinates $z_{i_1},\ldots,z_{i_k}$ vanishing, then
$F_{i_1}\cap\cdots\cap F_{i_k}\ne\varnothing$. Hence the claim
follows from the definition of $U(\sK)$ and $\sK$, see~\eqref{uk}
and~\eqref{simpo}.

(b) One needs to check that the condition of $P$ being simple
translates into the condition of~\eqref{zpqua} being
nondegenerate.

(c) We have the quotient projection $U(\sK)\to U(\sK)/R$ by the
action of $R$~\eqref{rsigma}, and both compact subsets $\zk\subset
U(\sK)$, $\zp\subset U(\sK)$ intersect each $R$-orbit at a single
point. For $\zk$ this is proved in~\cite[Theorem~2.2]{pa-us12},
and for $\zp$ in~\cite[Chapter~VI, Proposition~3.1.1]{audi91}.
Furthermore, $\zp$ is transverse to each orbit.
\end{proof}


\section{Complex structures}
\label{_complex_ana_stru_Section_}

Bosio and Meersseman~\cite{bo-me06} showed that some moment-angle
manifolds admit complex structures. They only considered
moment-angle manifolds corresponding to convex polytopes, and
identified them with \emph{LVM-manifolds}~\cite{meer00,me-ve04} (a
class of non-K\"ahler complex manifolds whose underlying smooth
manifolds are intersections of Hermitian quadrics in a complex
projective space). A modification of this construction was
considered in~\cite{pa-us12}, where a complex structure was
defined on any moment-angle manifold $\zk$ corresponding to a
complete simplicial fan. Similar results were obtained by
Tambour~\cite{tamb12}.

In this section we assume that $m-n$ is even. We can always
achieve this by adding a ghost vertex with any corresponding
vector to our data $\{\sK;\mb a_1,\ldots,\mb a_m\}$; topologically
this results in multiplying $\zk$ by a circle. We set
$\ell=\frac{m-n}2$.

We identify $\C^m$ (as a real vector space) with $\R^{2m}$ using
the map
\[
  (z_1,\ldots,z_m)\mapsto(x_1,y_1,\ldots,x_m,y_m),
\]
where $z_k=x_k+iy_k$, and consider the $\R$-linear map
\[
  \Re\colon\C^m\to\R^m,\qquad (z_1,\ldots,z_m)\mapsto(x_1,\ldots,x_m).
\]

In order to define a complex structure on the quotient $\zk\cong
U(\sK)/R$ we replace the action of $R$ by the action of a
holomorphic subgroup $C\subset(\C^\times)^m$ by means of the
following construction.

\begin{construction}\label{psi}
Let $\mb a_1,\ldots,\mb a_m$ be a configuration of vectors that
span $N_\R\cong\R^n$. Assume further that $m-n=2\ell$ is even.
Some of the $\mb a_i$'s may be zero. Consider the map
\[
  A\colon\R^m\to N_\R, \quad \mb e_i\mapsto\mb a_i.
\]

We choose a complex $\ell$-dimensional subspace in $\C^m$ which
projects isomorphically onto the real $(m-n)$-dimensional subspace
$\Ker A\subset\R^m$.  More precisely, choose a $\C$-linear map
$\varPsi\colon \C^\ell\to\C^m$ satisfying the two conditions:
\begin{itemize}
\item[(a)] the composite map
$\C^\ell\stackrel{\varPsi}\longrightarrow\C^m
\stackrel{\Re}\longrightarrow\R^m$ is a monomorphism;

\item[(b)] the composite map
$\C^\ell\stackrel{\varPsi}\longrightarrow\C^m
\stackrel{\Re}\longrightarrow\R^m \stackrel{A}\longrightarrow
N_\R$ is zero.
\end{itemize}
Set $\mathfrak c=\varPsi(\C^l)$. Then the two conditions above are
equivalent to the following:
\begin{itemize}
\item[$\mathrm{(a')}$] $\mathfrak c\cap\overline{\mathfrak c}=\{\mathbf0\}$;
\item[$\mathrm{(b')}$] $\mathfrak c\subset\Ker(A_\C\colon\C^m\to N_\C)$,
\end{itemize}
where $\overline{\mathfrak c}$ is the complex conjugate space and
$A_\C\colon\C^m\to N_\C$ is the complexification of the real map
$A\colon\R^m\to N_\R$. Consider the following commutative diagram:
\begin{equation}\label{cdiag}
\begin{CD}
  \C^\ell @>\varPsi>> \C^m @>\Re>> \R^m @>A>> N_\R\\
  @. @VV\exp V @VV\exp V\\
  \ @. (\C^\times )^m@>|\,\cdot|>> \R^m_>
\end{CD}
\end{equation}
where the vertical arrows are the componentwise exponential maps,
and $|\cdot|$ denotes the map
$(z_1,\ldots,z_m)\mapsto(|z_1|,\ldots,|z_m|)$. Now set
\begin{equation}\label{csigma}
  C=\exp\mathfrak c
  =\bigl\{(e^{w_1},\ldots,e^{w_m})\in(\C^\times)^m\colon \mb
  w=(w_1,\ldots,w_m)\in\varPsi(\C^\ell)\bigr\}.
\end{equation}
Then $C\cong\C^\ell$ is a complex (but not algebraic) subgroup
in~$(\C^\times)^m$, and $\mathfrak c$ is its Lie algebra. There is
a holomorphic action of $C$ on $\C^m$ and $U(\sK)$ by restriction.
Furthermore, condition~$\mathrm{(a')}$ above implies that $C$ is
closed in~$(\C^\times)^m$.
\end{construction}


\begin{example}\label{2torus}
Let $\mb a_1,\ldots,\mb a_m$ be the configuration of $m=2\ell$
zero vectors. We supplement it by the empty simplicial complex
$\sK$ on $[m]$ (with $m$ ghost vertices), so that the data
$\{\sK;\mb a_1,\ldots,\mb a_m\}$ define a complete fan in
0-dimensional space. Then $A\colon\R^m\to\R^0$ is a zero map, and
condition~(b) of Construction~\ref{psi} is void. Condition~(a)
means that $\C^\ell\stackrel{\varPsi}\longrightarrow\C^{2\ell}
\stackrel{\mathrm{Re}}\longrightarrow\R^{2\ell}$ is an isomorphism
of real spaces.

Consider the quotient $(\C^\times)^m/C$ (note that
$(\C^\times)^m=U(\sK)$ in our case). The exponential map
$\C^m\to(\C^\times)^m$ identifies $(\C^\times)^m$ with the
quotient of $\C^m$ by the imaginary lattice $\Gamma=\Z\langle2\pi
i\mb e_1,\ldots,2\pi i\mb e_m\rangle$. Condition~(a) implies that
the projection $p\colon\C^m\to\C^m/\mathfrak c$ is nondegenerate
on the imaginary subspace of~$\C^m$. In particular, $p\,(\Gamma)$
is a lattice of rank $m=2\ell$ in $\C^m/\mathfrak c\cong\C^\ell$.
Therefore,
\[
  (\C^\times)^m/C\cong\bigl(\C^m/\Gamma\bigr)/\mathfrak c
  =\bigl(\C^m/\mathfrak c\bigr)
  \big/p\,(\Gamma)\cong\C^\ell/\Z^{2\ell}
\]
is a complex compact $\ell$-dimensional torus.

Any complex torus can be obtained in this way. Indeed, let
$\varPsi\colon\C^\ell\to\C^m$ be given by an
$2\ell\times\ell$-matrix $\begin{pmatrix}-B\\I\end{pmatrix}$ where
$I$ is the unit matrix and $B$ is a square matrix of size~$\ell$.
Then $p\colon\C^m\to\C^m/\mathfrak c$ is given by the matrix
$(I\,B)$ in appropriate bases, and $(\C^\times)^m/C$ is isomorphic
to the quotient of $\C^\ell$ by the lattice $\Z\langle\mb
e_1,\ldots,\mb e_\ell,\mb b_1,\ldots,\mb b_\ell\rangle$, where
$\mb b_k$ is the $k$th column of~$B$. (Condition~(b) implies that
the imaginary part of $B$ is nondegenerate.)

For example, if $\ell=1$, then $\varPsi\colon\C\to\C^2$ is given
by $w\mapsto(\beta w,w)$ for some $\beta\in\C$, so that the
subgroup~\eqref{csigma} is
\[
  C=\{(e^{\beta w},e^w)\}\subset(\C^\times )^2.
\]
Condition~(a) implies that $\beta\notin\R$. Then
$\exp\varPsi\colon\C\to(\C^\times )^2$ is an embedding, and
\[
  (\C^\times )^2/C\cong\C/(\Z\oplus\beta\Z)
\]
is a complex 1-dimensional torus with the lattice parameter
$\beta\in\C$.
\end{example}

\begin{theorem}[{\cite[Theorem~3.3]{pa-us12}}]\label{zkcomplex}\
Assume that data $\{\sK;\mb a_1,\ldots,\mb a_m\}$ define a
complete fan~$\Sigma$ in~$N_\R\cong\R^n$, and $m-n=2\ell$. Let
$C\cong\C^\ell$ be the group defined by~\eqref{csigma}. Then
\begin{itemize}
\item[(a)]
the holomorphic action of $C$ on $U(\sK)$ is free and proper, and
the quotient $U(\sK)/C$ has the structure of a compact complex
manifold;

\item[(b)] $U(\sK)/C$ is diffeomorphic to~$\zk$.
\end{itemize}
Therefore, $\zk$ has a complex structure, in which each element of
$\T^m$ acts by a holomorphic transformation.
\end{theorem}

Besides compact complex tori described in Example~\ref{2torus},
other examples of $\zk$ include Hopf and Calabi--Eckmann manifolds
(see~\cite{pa-us12} and Example~\ref{hopf} below).

\begin{remark}
The subgroup $C\subset(\C^\times)^m$ depends on the vectors $\mb
a_1,\ldots,\mb a_m$ and the choice of~$\varPsi$ in
Construction~\ref{psi}. Unlike the smooth case, the complex
structure on $\zk$ depends in an essential way on all these pieces
of data.
\end{remark}

As in the real situation of Theorem~\ref{Rzkzp}, in the case of
normal fans we can view $\zk$ as the intersection of
quadrics~$\zp$ given by~\eqref{zpqua}:

\begin{theorem}\label{zpcom}
Let $P$ be a simple polytope given by~\eqref{ptope} with even
$m-n$, and let $\sK=\sK_P$ be the corresponding simplicial
complex~\eqref{simpo}. Then the composition
\[
  \zp\stackrel{i_{\mathcal Z}}\longrightarrow U(\sK)\to U(\sK)/C
\]
is a $\T^m$-equivariant diffeomorphism. It endows the intersection
of quadrics $\zp$ with the structure of a complex manifold.
\end{theorem}
\begin{proof}
We need to show that the $C$-orbit $C\mb z$ of any $\mb z\in
U(\sK)$ intersects $\zp\subset U(\sK)$ transversely at a single
point. First we show that the $C$-orbit of any $\mb y\in
U(\sK)/\T^m$ intersects $\zp/\T^m=i_P(P)$ at a single point; this
follows from Theorem~\ref{Rzkzp} and the fact that the induced
actions of $C$ and $R$ on $U(\sK)/\T^m$ coincide. Then we show
that $C\mb z$ intersects the full preimage $\zp=\mu^{-1}(i_P(P))$
at a single point using the fact that $C$ and $\T^m$ have trivial
intersection in $(\C^\times)^m$. The transversality of the
intersection $C\mb z\cap \zp$ follows from the transversality for
$R$-orbits, because $T_{\mb z}(R\mb z)\oplus T_{\mb z}\zp=T_{\mb
z}(C\mb z)\oplus T_{\mb z}\zp$.
\end{proof}

\begin{remark}
The embedding $i_{\mathcal Z}\colon\zp\to U(\sK)$ is not
holomorphic. In the polytopal case, the complex structure on $\zp$
coming from Theorem~\ref{zpcom} is equivalent to the structure of
an LVM-manifold described in~\cite{bo-me06}.
\end{remark}

\section{Submanifolds, analytic subsets and meromorphic functions}
In this section we consider moment-angle manifolds $\zk$
corresponding to complete simplicial $n$-dimensional fans $\Sigma$
defined by data $\{\sK;\mb a_1,\ldots,\mb a_m\}$ with $m-n=2\ell$.
The manifold $\zk$ is diffeomorphic to the quotient $U(\sK)/C$, as
described in the previous section. This is used to equip $\zk$
with a complex structure. The complex structure depends on the
vectors $\mb a_1,\ldots,\mb a_m$ and the choice of a map $\varPsi$
in Construction~\ref{psi}, but we shall not reflect this in the
notation.

\subsection{Coordinate submanifolds}
For each $J\subset [m]$, we define the corresponding
\emph{coordinate submanifold} in $\zk$ by
\[
  \mathcal Z_{\mathcal K_J}=\{(z_1,\ldots,z_m)\in\zk\colon z_i=0\quad\text{for }i\notin
  J\}.
\]
It is the moment-angle manifold corresponding to the \emph{full
subcomplex}
\[
  \sK_J=\{I\in\sK\colon I\subset J\}.
\]
In the situation of Theorem~\ref{zkcomplex},  $\mathcal
Z_{\mathcal K_J}$ is identified with the quotient of
\[
  U(\sK_J)=\{(z_1,\ldots,z_m)\in U(\sK)\colon z_i=0\quad\text{for }i\notin
  J\}
\]
by the group $C$. In particular, $U(\sK_J)/C$ is a complex
submanifold in~$U(\sK)/C$.

Observe that the closure of any $(\C^\times)^m$-orbit of $U(\sK)$
has the form $U(\sK_J)$ for some $J\subset[m]$ (in particular, the
dense orbit corresponds to $J=[m]$). Similarly, the closure of any
$(\C^\times)^m/C$-orbit of $\zk\cong U(\sK)/C$ has the
form~$\mathcal Z_{\mathcal K_J}$.

\subsection{Holomorphic foliations}
\begin{construction}[Holomorphic foliation $\cF$ on
$\zk$]\label{Ffoli} Consider the complexified map
$A_\C\colon\C^m\to N_\C$, $\mb e_i\mapsto\mb a_i$. Define the
following complex $(m-n)$-dimensional subgroup in~$(\C^\times)^m$:
\begin{equation}\label{gsigma}
  G=\exp(\Ker A_\C)=\bigl\{\bigl(e^{z_1},\ldots,e^{z_m}\bigr)\in
  (\C^\times)^m\colon(z_1,\ldots,z_m)\in\Ker A_\C\bigr\}.
\end{equation}
It contains both the real subgroup~$R$~\eqref{rsigma} and the
complex subgroup $C$~\eqref{csigma}.

The group $G$ acts on $U(\sK)$, and its orbits define a
holomorphic foliation on~$U(\sK)$. Since $G\subset(\C^\times)^m$,
this action is free on the open subset $(\C^\times)^m\subset
U(\sK)$, so that a generic leaf of the foliation has complex
dimension $m-n=2\ell$ (in fact, all leaves have the same
dimension, see Proposition~\ref{leaves}). The $\ell$-dimensional
closed subgroup $C\subset G$ acts on $U(\sK)$ freely and properly
by Theorem~\ref{zkcomplex}, so that $U(\sK)/C$ carries a
holomorphic action of the quotient group $D=G/C$.

Recall from Construction~\ref{psi} that $\mathfrak
c=\mathop{\mathrm{Lie}}C\subset\Ker
A_\C=\mathop{\mathrm{Lie}}G\subset\C^m$.

The action of the group $D$ on $\zk$ is holomorphic and
nondegenerate, i.e., for any point $x\in \zk$ the differential
$T_e D\to T_x \zk$ of the action map $g\mapsto g\cdot x$ is
injective. Therefore the orbits of~$D$ define a smooth holomorphic
foliation $\cF$ on $\zk$.
\end{construction}

The leaves of the foliations from Construction~\ref{Ffoli} are
described as follows (compare Example~\ref{2torus}):

\begin{proposition}\label{leaves}
For any subset $I\subset [m]$, define the coordinate subspace
$\C^I=\C\langle\mb e_k\colon k\in I\rangle\subset\C^m$ and the
subgroup
\[
  \Gamma_I=\Ker A_\C\cap\bigl(\Z\langle2\pi i\mb e_1,\ldots,2\pi i
  \mb e_m\rangle+\C^I\bigr).
\]
\begin{itemize}
\item[(a)] $\Gamma_I$ is discrete whenever $I\in\sK$.

\item[(b)] We have $G\cong\Ker A_\C/\Gamma$, where $\Gamma=\Gamma_\varnothing$.
Any leaf $G\mb z$ of the foliation on $U(\sK)$ by the orbits of
$G$ is biholomorphic to
$(\C^\times)^{\rk\Gamma_I}\times\C^{2\ell-\rk\Gamma_I}$, where
$I\in\sK$ is the set of zero coordinates of $\mb z\in U(\sK)$.

\item[(c)] We have $D=G/C\cong\bigl(\Ker A_\C/\mathfrak c\bigr)\big/ p\,(\Gamma)$,
where $p\colon\Ker A_\C\to\Ker A_\C/\mathfrak c$ is the
projection. Any leaf $\cF z$ of the foliation $\cF$ on the
moment-angle manifold $\zk\cong U(\sK)/C$ is biholomorphic to
$\C^\ell/p(\Gamma_I)$.
\end{itemize}
\end{proposition}
\begin{proof} (a) If $I\in\sK$, then the composite map $\C^I\to\C^m\stackrel{A_\C}{\longrightarrow}N_\C$
is monomorphic by the definition of a simplicial fan. Therefore,
$\C^I\cap\Ker A_\C=\{\bf0\}$, which implies that $\Gamma_I$ is
discrete.

(b) Since $(\C^\times)^m=\exp(\C^m)\cong\C^m/\Z\langle2\pi i\mb
e_1,\ldots,2\pi i \mb e_m\rangle$, the isomorphism $G\cong\Ker
A_\C/\Gamma$ follows from the definition of~$G$,
see~\eqref{gsigma}. Therefore, to describe the orbit $G\mb z$, we
can consider the orbit of the exponential action of $\Ker A_\C$
instead. The stabiliser of $\mb z$ under the action of $\Ker A_\C$
is exactly $\Gamma_I$, and the orbit itself is $\Ker
A_\C/\Gamma_I\cong(\C^\times)^{\rk\Gamma_I}\times\C^{2\ell-\rk\Gamma_I}$.

(c) By the same argument, the orbit of $z\in\zk$ under the action
of $D=G/C$ is $(\Ker A_\C/\mathfrak c)/p(\Gamma_I)$, and $\Ker
A_\C/\mathfrak c\cong\C^\ell$.
\end{proof}

A subset $X'\subset X$ of a space with Lebesgue measure $(X,\mu)$
is said to contain \emph{almost all} elements of $X$ if its
complement has zero measure: $\mu(X\backslash X')=0$; in this case
points of $X'$ are \emph{generic} for~$X$, and the condition
specifying $X'$ in $X$ is \emph{generic}.

A vector of $\Q^m\subset\R^m$ is called \emph{rational}. A linear
subspace $V\subset\R^m$ is \emph{rational} if it is generated by
rational vectors. A linear function $\phi\colon\R^m\to\R$ is
\emph{rational} if it takes rational values on rational vectors
(equivalently, if it has rational coefficients when written in the
standard basis of~$(\R^m)^*$).

Here are two examples of generic conditions for linear maps
$A\colon \R^m\to N_\R$
(assuming that $m>\dim N_\R$, which is the case when $\Sigma$ is a
complete fan):
\begin{itemize}
\item[(g1)] no rational linear function on $\R^m$ vanishes identically on $\Ker A$ (equivalently,
$\Ker A$ is not contained in any rational hyperplane of~$\R^m$);
\item[(g2)] $\Ker A$ does not contain rational vectors of~$\R^m$.
\end{itemize}

We observe that if the configuration of vectors $\mb
a_1,\ldots,\mb a_m$ satisfies the generic condition~(g2), then the
subgroup $\Gamma$ of Proposition~\ref{leaves} is trivial, $G$ is
biholomorphic to $\C^{2\ell}$ and $D=G/C$ is biholomorphic
to~$\C^\ell$.

On the other hand, if $\Ker A\subset\R^m$ is a rational subspace
(i.e. it has a basis consisting of vectors with integer
coordinates), then $\rk\Gamma=2\ell$ and the subgroup
$G\subset(\C^\times)^m$ is closed and is isomorphic to
$(\C^\times)^{2\ell}$. In this case the vectors $\mb
a_1,\ldots,\mb a_m$ generate a lattice $N_\Z=\Z\langle\mb
a_1,\ldots,\mb a_m\rangle$, and $\Sigma$ is a rational (possibly
singular) fan with respect to this lattice. If each vector $\mb
a_i$ is primitive in~$N_\Z$, then the quotient $U(\sK)/G$ is the
\emph{toric variety} $V_\Sigma$ corresponding to the fan~$\Sigma$
(see e.g.~\cite{c-l-s11}). The foliation $\cF$ gives rise to a
holomorphic principal Seifert bundle $\pi\colon\zk\to V_\Sigma$
with fibres compact complex tori $G/C$ (leaves of~$\cF$),
see~\cite{me-ve04} and~\cite[Proposition~5.2]{pa-us12}.

Foliations similar to~$\mathcal F$ were studied in~\cite{ba-za},
together with generic conditions.

\subsection{Transverse K\"ahler forms}
\begin{definition}
Let $\cF$ be a holomorphic $\ell$-dimensional
foliation on a complex manifold~$M$. A $(1,1)$-form $\omega_\cF$
on $M$ is called \emph{transverse K\"ahler} with respect to~$\cF$
if the following two conditions are satisfied:
\begin{itemize}
\item[(a)] $\omega_\cF$ is closed, i.e. $d\omega_\cF=0$;

\item[(b)] $\omega_\cF$ is positive and the zero space of
$\omega_\cF$ is the tangent space of $\cF$. (That is,
$\omega(V,JV)\ge0$ for any real tangent vector~$V$, and
$\omega_\cF(V,JV)=0$ if and only if $V$ is tangent to~$\cF$; here
$J$ is the operator of complex structure.)
\end{itemize}
\end{definition}

One way to define a transverse K\"ahler form on $\zk$ is to use a
modification of an argument of Loeb and Nicolau~\cite{lo-ni99}; it
works only for normal fans:

\begin{proposition}\label{trkah}
Assume that $\Sigma=\Sigma_P$ is the normal fan of a simple
polytope~$P$. Then the foliation~$\cF$ described in
Construction~{\rm\ref{Ffoli}} admits a transverse K\"ahler
form~$\omega_\cF$.
\end{proposition}
\begin{proof}
Since $\Sigma$ is a normal fan of~$P$, we have a $\mathbb
T^m$-equivariant diffeomorphism $\phi\colon
U(\sK)/C\stackrel\cong\longrightarrow \zp$ between $\zk=U(\sK)/C$
and the intersection of quadrics $\zp\subset \C^m$
(Theorem~\ref{zpcom}). Let $\omega=\frac i2\sum dz_k\wedge
d\overline{z}_k$ be the standard form on~$\C^m$, and
$\omega_{\mathcal Z}=i^*_{\mathcal Z}\omega$ its restriction to
the intersection of quadrics~$\zp$. Define
$\omega_\cF=\varphi^*\omega_{\mathcal Z}$. Using the same argument
as~\cite[Proposition~2]{lo-ni99} one verifies that the zero
foliation of the form $\omega_{\mathcal Z}$ coincides with
$\varphi(\mathcal F)$, and therefore $\omega_\cF$ is transverse
K\"ahler.
\end{proof}

The condition that $\Sigma$ is a normal fan in
Proposition~\ref{trkah} is important. Indeed, it was shown
in~\cite{cu-za07} that general \emph{LVMB-manifolds} do not admit
transverse K\"ahler forms. A similar argument proves that there
are no transverse K\"ahler forms on the quotients $U(\sK)/C$ for
general complete simplicial fans.

We can relax the condition on the fan slightly at the cost of
weakening the conditions on the form~$\omega_\cF$:

\begin{definition}
A complete simplicial fan $\Sigma$ in $N_\R\cong\R^n$ is called
\emph{weakly normal} if there exists a (not necessarily simple)
$n$-dimensional polytope $P$ given by~\eqref{ptope} such that
$\Sigma$ is a simplicial subdivision of the normal fan~$\Sigma_P$.
\end{definition}

\begin{remark}
Let $\Sigma$ be a complete simplicial fan with chosen generators
$\mb a_1,\ldots,\mb a_m$ of its edges. Then $\Sigma$ is normal if
one can find constants $b_1,\ldots,b_m$ such the polytope $P$
defined by~\eqref{ptope} is simple and the simplicial complex
$\sK_P$~\eqref{simpo} coincides with the underlying complex $\sK$
of the fan. A fan $\Sigma$ is weakly normal if one can find
$b_1,\ldots,b_m$ so that $\sK$ is contained in~$\sK_P$;
equivalently, $F_{i_1}\cap\cdots\cap F_{i_k}\ne\varnothing$ in~$P$
if $\mb a_{i_1},\ldots,\mb a_{i_k}$ span a cone of~$\Sigma$.

In the geometry of toric varieties, a lattice
polytope~\eqref{ptope} (in which $\mb a_1,\ldots,\mb a_m$ are
primitive lattice vectors, and $b_1,\ldots,b_m$ are integers)
gives rise to an ample divisor on the toric variety $V_\Sigma$
corresponding to the normal fan $\Sigma=\Sigma_P$. In particular,
toric varieties corresponding to (rational) normal fans are
projective. Weakly normal rational fans $\Sigma$ give rise to
\emph{semiample divisors} on their corresponding toric varieties
(see~\cite[\S{}II.4.2]{a-d-h-l13}); such a divisor defines a map
from $V_\Sigma$ to a projective space, which is not necessarily an
embedding.

Given a cone $\sigma_I\in \Sigma$, the \emph{quotient fan}
$\Sigma/\sigma_I$ is defined. Its associated simplicial complex is
the link $\mathop{\mathrm{lk}}_{\sK}I$. If $\Sigma$ is weakly
normal (with the corresponding polytope~$P$), then
$\Sigma/\sigma_I$ is also weakly normal (with the corresponding
polytope being a face of~$P$). This fact will be used below in the
inductive arguments for Theorem~\ref{kahlsub} and
Theorem~\ref{final}.
\end{remark}

\begin{theorem}\label{form2}
Assume that $\Sigma$ is a weakly normal fan defined by $\{\mathcal
K;\mb a_1,\ldots,\mb a_m\}$. Then there exists an exact
$(1,1)$-form~$\omega_\cF$ on $\zk=U(\sK)/C$ which is transverse
K\"ahler for the foliation~$\cF$ on the dense open subset
$(\C^\times)^m/C\subset U(\sK)/C$.
\end{theorem}
\begin{proof}
The plan of the proof is as follows. We construct a smooth
function $f\colon U(\sK)\to\R$ which is plurisubharmonic, so that
$\omega=dd^cf$ is a positive $(1,1)$-form (here $d=\partial
+\bar\partial$ and $d^c=-i(\partial-\bar\partial)$). We check that
the kernel of $\omega$ consists of tangents to the orbits of~$G$.
Then we show that $\omega$ descends to a form $\omega_\cF$ on
$U(\sK)/C$ with the required properties.

We consider the short exact sequence
\[
  0\to\Ker A\longrightarrow\R^m\stackrel{A}\longrightarrow
  N_\R\to0
\]
where $A$ is given by $\mb e_i\mapsto\mb a_i$. Since $\Sigma$ is
weakly normal, there is the corresponding polytope~\eqref{ptope}.
We think of $(b_1,\ldots,b_m)$ as a linear function $\mb
b\colon\R^m\to\R$, and denote by $\chi_{\mb b}\colon\Ker A\to\R$
the restriction of $\mb b$ to $\Ker A$.

Let $I\in\sK$ be a maximal simplex (i.e. the vectors $\mb a_i$,
$i\in I$, span a maximal cone of~$\Sigma$), and let $\mb
u_I=\bigcap_{i\in I}F_i$ be the vertex of $P$ corresponding
to~$I$. Define the linear function $\beta_I\in(\R^m)^*$ whose
coordinates in the standard basis are $\langle\mb a_i,\mb
u_I\rangle+b_i$, $i=1,\ldots,m$. It follows that all coordinates
of $\beta_I$ are nonnegative, and the coordinates corresponding to
$i\in I$ are zero. (Some other coordinates of $\beta_I$ may also
vanish, since the polytope $P$ is not necessarily simple and
different $I$ may give the same vertex.) Also, the restriction of
$\beta_I$ to $\Ker A$ is $\chi_{\mb b}$ because $A^*\colon
N_\R^*\to(\R^m)^*$ is given by $\mb u\mapsto(\langle\mb a_1,\mb
u\rangle,\ldots,\langle\mb a_m,\mb u\rangle)$. Finally, by
multiplying all $b_i$ simultaneously by a positive factor we can
obtain that all coordinates of all $\beta_I$ are either zero or
$\ge2$.

We define the function $f\colon U(\sK)\to\R$ as follows:
\[
  f(\mb z\,)=\log \Bigl(\sum_{\text{maximal }I\in\sK}
  |\mb z\,|^{\beta_I}\Bigr)
\]
where $|\mb z\,|^\alpha=|z_1|^{\alpha(\mb e_1)}\cdots
|z_m|^{\alpha(\mb e_m)}$ is the monomial corresponding to a linear
function $\alpha\in(\R^m)^*$, and we set $0^0=1$. By definition of
$U(\sK)$, the set of zero coordinates of any point $\mb z\in
U(\sK)$ is contained in a maximal simplex $I\in\sK$, hence $|\mb
z\,|^{\beta_I}>0$. Therefore, the function $f$ is smooth
on~$U(\sK)$.

Now define the real $(1,1)$-form $\omega=dd^cf$ on $U(\sK)$.
By~\cite[Theorem~I.5.6]{dema12}, the function $f$ is
plurisubharmonic, so that $\omega$ is positive.

\begin{lemma}\label{kerom}
The kernel of $\omega|_{(\C^\times)^m}$ consists of tangent spaces
to the orbits of the action of $G$, see~\eqref{gsigma}.
\end{lemma}
\begin{proof}
Let $J$ be the operator of complex structure. Since $\omega$ is
positive, its kernel coincides with the kernel of the symmetric
2-form $\omega(\,\cdot\,,\,J\cdot\,)$.

Take $\mb z\in(\C^\times)^m$. By writing $\mb z$ in polar
coordinates, $z_k=\rho_ke^{i\varphi_k}$, we decompose the real
tangent space to $(\C^\times)^m$ at $\mb z$ as $T_{\mb
z}=T_{\rho}\oplus T_{\varphi}$, where $T_{\rho}$ and $T_{\varphi}$
consist of tangents to radial and angular directions respectively.
Since $f$ does not depend on the~$\varphi_i$'s, the matrix of
$\omega(\,\cdot\,,\,J\cdot\,)$ is block-diagonal with respect to
the decomposition $T_{\mb z}=T_{\rho}\oplus T_{\varphi}$. The
diagonal blocks are identical since $\omega$ is $J$-invariant. It
follows that $\Ker\omega=(\Ker\omega\cap T_{\rho})\oplus
J(\Ker\omega\cap T_{\rho})$. It remains to describe
$\Ker\omega\cap T_\rho$. To do this, we identify $T_\rho$ with the
Lie algebra $\R^m$ of the group $\R^m_>$ acting on $(\C^\times)^m$
by coordinatewise multiplications.

Take a radial vector field $V\in T_\rho$ corresponding to a
1-parameter subgroup
$t\mapsto(e^{\lambda_1t}z_1,\ldots,e^{\lambda_mt}z_m)$ where
$\lambda=(\lambda_1,\ldots,\lambda_m)\in\R^m$. Then
\begin{multline*}
  \omega(V,JV)=(dd^c f)(V,JV)=
  L_V\langle d^c f,JV\rangle-L_{J\!V}\langle d^c f,V\rangle\\
  =L_{J\!V}\langle df,JV\rangle-
  L_V\langle df,J^2V\rangle
  =\frac{d^2f(e^{\lambda_1t}z_1,\ldots,e^{\lambda_mt}z_m)}{dt^2}
  \Big|_{t=0},
\end{multline*}
where $L_V$ denotes the Lie derivative along~$V$, the second
equality holds because $V$ and $JV$ commute, and $\langle
df,JV\rangle=0$ because $f$ does not depend on angular
coordinates. It remains to calculate the second derivative.
Note that for any linear function $\beta_I\in(\R^m)^*$,  we have
$|\mb z(t)|^{\beta_I}=e^{\langle\beta_I,\lambda\rangle t} |\mb
z|^{\beta_I}$ along the curve $t\mapsto\mb z(t)=
(e^{\lambda_1t}z_1,\ldots,e^{\lambda_mt}z_m)$. Therefore,
\begin{multline*}
  \frac{d^2}{dt^2}f(e^{\lambda_1t}z_1,\ldots,e^{\lambda_mt}z_m)\bigg|_{t=0}=
  \frac d{dt}\biggl(\frac{\sum_I\langle\beta_I,\lambda\rangle e^{\langle\beta_I,\lambda\rangle t} |\mb z|^{\beta_I}}
  {\sum_I e^{\langle\beta_I,\lambda\rangle t} |\mb z|^{\beta_I}}\biggr)\bigg|_{t=0}\\
  =\frac1{\bigl(\sum_I|\mb z\,|^{\beta_I}\bigr)^2}
  \biggl(\sum_I\langle\beta_I,\lambda\rangle^2|\mb z\,|^{\beta_I}\cdot
  \sum_J|\mb z\,|^{\beta_J}
  -\Bigl(\sum_I\langle\beta_I,\lambda\rangle|\mb z\,|^{\beta_I}
  \Bigr)^2\biggr)\\
  =\frac1{\bigl(\sum_I|\mb z\,|^{\beta_I}\bigr)^2}
  \biggl(\sum_{I,J}|\mb z\,|^{\beta_I}|\mb z\,|^{\beta_J}
  \Bigl(\langle\beta_I,\lambda\rangle-
  \langle\beta_J,\lambda\rangle\Bigr)^2\biggr).
\end{multline*}
We claim that this vanishes precisely when $\lambda\in\Ker A$.

If $\lambda\in\Ker A$ then
$\langle\beta_I,\lambda\rangle=\sum_{i=1}^m\lambda_ib_i$ by the
definition of~$\beta_I$, and this is independent of~$I$.
Therefore, the last sum in the displayed formula above vanishes.

Conversely, if the sum above is zero for a given $\lambda\in\R^m$,
then $\langle\beta_I-\beta_J,\lambda\rangle=0$ for any pair of
maximal simplices $I,J\in\sK$ (here we use the fact that $\mb
z\in(\C^\times)^m$). We have $\beta_I-\beta_J=A^*(\mb u_I-\mb
u_J)$, where $\mb u_I-\mb u_J$ is the vector connecting the
vertices $\mb u_I$ and $\mb u_J$ of~$P$. Since $P$ is
$n$-dimensional, the linear span of all vectors $\beta_I-\beta_J$
is the whole $A^*(N^*_\R)$. Thus, $\lambda\in\Ker A$.

We have therefore identified $\Ker\omega\cap T_\rho$ with $\Ker
A\subset\R^m$. On the other hand, $\Ker A$ is the tangent space to
the orbits of $R\subset G$, see~\eqref{rsigma}. Since
$\Ker\omega=(\Ker\omega\cap T_{\rho})\oplus J(\Ker\omega\cap
T_{\rho})$, we can identify $\Ker\omega$ with $\Ker A\oplus J\Ker
A$, which is exactly the tangent space to an orbit of~$G$.
\end{proof}

Now we can finish the proof of Theorem~\ref{form2}. We need to
show that the form $\omega=dd^c f$ descends to a form on
$U(\sK)/C$. In other words, we need to show that $\omega$ is
\emph{basic} with respect to the foliation defined by the orbits
of~$C$, i.e. $\omega(V)=0$ and $L_V\omega=0$ for any vector field
$V$ tangent to $C$-orbits. Since $C\subset G$, the previous lemma
implies that $\Ker\omega$ contains~$V$, so $\omega(V)=0$.
By the Cartan formula, $L_V\omega = (d i_V+i_V
d)\omega=d(i_V\omega)=0$, since $d\omega=0$ and $i_V\omega =
\omega(V)=0$.

Let $\omega_\cF$ be the form obtained by descending $\omega$ to
$U(\sK)/C$. Then $\omega_\cF$ is positive since $\omega$ is
positive, and Lemma~\ref{kerom} implies that $\Ker\omega_{\cF}$
consists exactly of the tangents to the orbits of $G/C$. Thus,
$\omega_\cF$ is transverse K\"ahler for~$\cF$ on
$(\C^\times)^m/C$.

To see that the form $\omega_\cF$ is exact we note that the fibres
of the projection $p\colon U(\sK)\to U(\sK)/C$ are contractible
(as $C\cong\C^\ell$), so $p$ is a homotopy equivalence. Hence $p$
induces an isomorphism of the cohomology groups $p^*\colon
H^*(U(\sK)/C;\R)\stackrel\cong\longrightarrow H^*(U(\sK);\R)$.
Since $p^*\omega_{\cF}=\omega$ and $[\omega]=0$ in
$H^2(U(\sK);\R)$, the form $\omega_{\cF}$ is also exact. One can
actually construct a 1-form $\eta_\cF$ satisfying
$d\eta_\cF=\omega_{\cF}$ more explicitly as follows. Choose a
smooth section $s\colon U(\sK)/C\to U(\sK)$. For any $x\in
U(\sK)/C$ the restriction of $\omega$ to the fibre $p^{-1}(x)$ is
zero, hence $d^c f|_{p^{-1}(x)}$ is closed. Since $p^{-1}(x)$ is
contractible, $d^c f|_{p^{-1}(x)}$ is exact. Hence, $d^c
f|_{p^{-1}(x)}=d g_x$ for some function $g_x\colon p^{-1}(x)\to
\C$ defined up to a constant. We choose $g_x$ such that
$g_x(s(x))=0$. The collection of functions $\{g_x\}_x$ defines a
smooth function $g\colon U(\sK)\to \C$ such that for any $x\in
U(\sK)/C$ we have $dg|_{p^{-1}(x)}=d^c f|_{p^{-1}(x)}$. Then the
1-form $\eta=d^cf-dg$ is basic with respect to the foliation
generated by the fibers of $p$, since
$L_V\eta=(di_V\eta+i_Vd\eta)=i_V\omega=0$ for any vector field $V$
tangent to the fibre at~$p$. Therefore $\eta$ descends to a form
$\eta_\cF$ on $U(\sK)/C$ such that $d\eta_\cF=\omega_{\cF}$.
\end{proof}

\subsection{K\"ahler and Fujiki class $\mathcal C$ subvarieties}
Now we can use the transverse K\"ahler form $\omega_\cF$ to
describe complex submanifolds and analytic subsets in~$\zk$.

\begin{definition}
A compact complex variety $M$ is
said to be of \emph{Fujiki class $\mathcal C$} if it is
bimeromorphic to a K\"ahler manifold.
\end{definition}

\begin{remark}
Since a blow-up of a K\"ahler manifold is again K\"ahler, for each
Fujiki class~$\mathcal C$ variety $X$ there exists a birational
holomorphic map $\widetilde X \to X$, where $\widetilde X$ is a
compact K\"ahler manifold.
\end{remark}

\begin{definition}
Let $M$ be a compact complex manifold, and $\Theta$ a closed
$(1,1)$-current on $M$ (see Subection~\ref{sscurrents}).
Assume that $M$ admits a positive (1,1)-form $\eta$ such that the
current $\Theta -\eta$ is positive. Then $\Theta$ is called a
\emph{K\"ahler current}.
\end{definition}

\begin{theorem}[{\cite[Theorem~0.6]{de-pa04}}]
A compact complex manifold admits a K\"ahler current if and only
if it is of Fujiki class $\mathcal C$.
\end{theorem}

We will use a modification of this statement for Fujiki class $\mathcal C$
\emph{varieties} to prove the following:

%

\begin{theorem}\label{kahlsub}
Under the assumptions of Theorem~{\rm\ref{form2}}, any Fujiki
class $\mathcal C$ subvariety of $\zk= U(\sK)/C$ is contained in a 
leaf of the $\ell$-dimensional foliation~$\cF$.
\end{theorem}
\begin{proof}
Let $i\colon Y\hookrightarrow U(\sK)/C$ be a Fujiki class
$\mathcal C$ subvariety. We may assume that $Y$ is not contained
in a coordinate submanifold of $\zk$, i.\,e., $Y$ contains a point
from $(\C^\times)^m/C$. (Otherwise choose the smallest by
inclusion coordinate submanifold $\mathcal
Z_{\sK_J}=U(\sK_J)/C\subset U(\sK)/C$ such that $Y\subset\mathcal
Z_{\sK_J}$, and consider $\mathcal Z_{\sK_J}$ instead of~$\zk$.)

We claim that $Y$ is
contained in a $D$-orbit (a leaf of~$\cF$).
Suppose this is not the case. Then
\[
  k=\min_{y\in Y_{\rm reg}}\dim (T_y Y\cap T_y\cF)<\dim Y,
\]
and the minimum is achieved at a generic point $y\in Y$.

Consider a smooth K\"ahler birational modification
$(\widetilde{Y},\omega_{\widetilde{Y}})$ of $Y$ with a birational
map $p\colon \widetilde{Y}\to Y$. Let $\omega_\cF$ be the
transverse K\"ahler form constructed in Theorem~\ref{form2}, and
let $\widetilde\omega_\cF=p^*i^*\omega_\cF$ be its pullback to
$\widetilde{Y}$. Consider the top-degree differential form
\[
\nu=\widetilde\omega_\cF^{\dim Y-k}\wedge \omega^k_{\widetilde Y}
\]
on $\widetilde{Y}$. Then $\nu$ is a positive measure, since
$\omega_\cF$, $\widetilde\omega_\cF$ and $\omega_{\widetilde Y}$
are all positive (1,1)-forms. We claim that $\nu$ is nontrivial.
Take a generic $y\in Y_{\rm reg}$, such that $p$ is an isomorphism
in a neighbourhood of $p^{-1}(y)$ and $\dim (T_y Y\cap T_y\cF)=k$.
We claim that $\nu$ is nonzero (strictly positive) at $p^{-1}(y)$.
Indeed, $i^*\omega_\cF$ is a restriction of the transverse
K\"ahler form $\omega_\cF$, so it is strictly positive on a $(\dim
Y-k)$-dimensional subspace transverse to its $k$-dimensional null
space $T_y Y\cap T_y\cF\subset T_y Y$. Hence,
$\widetilde\omega_\cF$ is also positive on the corresponding
subspace of $T_{p^{-1}(y)}\widetilde Y$. Since
$\omega_{\widetilde{Y}}$ is K\"ahler and positive in all
directions, $\nu$ is strictly positive at $p^{-1}(y)$.

Now recall that $\widetilde\omega_\cF=d\alpha$ is exact, so by the Stokes formula,
\[
  \int_{\widetilde Y} \widetilde\omega_\cF^{\dim Y-k}\wedge \omega^k_{\widetilde Y}=
  \int_{\widetilde Y} d(\alpha\wedge\widetilde\omega_\cF^{\dim Y-k-1}\wedge \omega^k_{\widetilde Y})
  =0.
\]
On the other hand, the measure $\nu$ is positive and nonzero at a
point $p^{-1}(y)$, leading to a contradiction. Hence $k=\dim Y$
and $Y$ is contained in a leaf of~$\cF$.
%
\end{proof}

\begin{corollary}\label{leavescor}
Assume that the subspace $\Ker A\subset \R^m$ does not contain
rational vectors. Then there are no positive dimensional Fujiki
class $\mathcal{C}$ subvarieties through  a generic point
of~$\zk$. More precisely, there are no such subvarieties
intersecting nontrivially the open subset $(\C^\times)^m/C\subset
U(\sK)/C=\zk$.
\end{corollary}
\begin{proof}
If $\Ker A\subset \R^m$ does not contain rational vectors, then,
in the notation of Proposition~\ref{leaves}, the group $\Gamma$ is
trivial. Therefore, any leaf in the open part
$(\C^\times)^m/C\subset U(\sK)/C$ of the foliation $\cF$ is
isomorphic to~$\C^\ell$, and therefore cannot contain Fujiki
class~$\mathcal C$ subvarieties.
\end{proof}

\subsection{The case of 1-dimensional foliation}
This case was studied by Loeb and Nicolau~\cite{lo-ni99}. Here is
how their results translate into our setting:

\begin{theorem}\label{onedim}
Assume that the foliation $\cF$ is $1$-dimensional, i.e. $\ell=1$.
\begin{itemize}
\item[(a)]
If $\Ker A\subset\R^m$ is a rational subspace, then any analytic
subset of $\zk$ is either a point, or has the form $\pi^{-1}(X)$,
where $\pi\colon U(\sK)/C\to V$ is the principal Seifert
bundle over the variety $V=U(\sK)/G$ and $X\subset
V$ is a subvariety;

\item[(b)]
If no rational linear function on $\R^m$ vanishes on $\Ker A$,
then any irreducible analytic subset of $\zk$ is either a
coordinate submanifold or a point.
\end{itemize}
\end{theorem}
\begin{proof}
First observe that if $m-n=2\ell=2$, then the fan $\Sigma$ is
normal. (Indeed, the corresponding polytope $P$ can be obtained by
truncating an $n$-simplex at a vertex.) Therefore, by
Proposition~\ref{trkah}, there exists a transverse K\"ahler
form~$\omega_\cF$. Let $Y$ be an analytic subset  of positive
dimension in $\zk$. Then $Y$ consists of leaves of the foliation
$\cF$, as otherwise the integral $\int_Y \omega_\cF^{\dim Y}$ of
the exact form $\omega_\cF^{\dim Y}$ over $Y$ is positive. In
other words, $Y$ is $D$-invariant.

Under assumption~(a), both $G\subset (\C^\times)^m$ and
$D=G/C\subset (\C^\times)^m/C$ are closed subgroups and $Y$ has
the form $\pi^{-1}(Y/D)$, where $Y/D \subset U(\sK)/G = V$.

When each vector $\mb a_i$ is primitive in the lattice
$\Z\langle\mb a_1,\dots,\mb a_m\rangle$, the variety $V$ is the
toric variety $V_\Sigma$ corresponding to the rational
fan~$\Sigma$. In general, $V$ is a finite branched covering
over~$V_\Sigma$, see~\cite[Example 2.8]{me-ve04}.


Under assumption (b), we claim that the minimal closed complex
subgroup $\overline D$ containing $D$ is the whole
$(\C^\times)^m/C$. This claim is equivalent to that $\overline
G=(\C^\times)^m$. Indeed, since $\overline G$ is closed, the
intersection $\mathop{\mathrm{Lie}}\overline G\cap i\R^m$ is a
rational subspace. Now if $\mathop{\mathrm{Lie}}\overline G\cap
i\R^m\ne i\R^m$, then there exists a rational linear function
vanishing on $\mathop{\mathrm{Lie}}\overline G$ and therefore on
$\mathop{\mathrm{Lie}}G=\Ker A_\C$, leading to a contradiction.
Hence $\mathop{\mathrm{Lie}}\overline G\cap i\R^m= i\R^m$, i.e.
$\mathop{\mathrm{Lie}}\overline G=\C^m$, and the claim is proved.
It follows that the subset $Y$ is $(\C^\times)^m/C$-invariant. An
irreducible $(\C^\times)^m/C$-invariant analytic subset of
$U(\sK)/C$ is a coordinate submanifold.
\end{proof}


\begin{example}[Hopf manifold]\label{hopf}
Let $\mb a_1,\ldots,\mb a_{n+1}$ be a set of vectors which span
$N_\R\cong\R^n$ and satisfy a linear relation $\lambda_1\mb
a_1+\cdots+\lambda_{n+1}\mb a_{n+1}=\mathbf0$ with all
$\lambda_k>0$. Let $\Sigma$ be the complete simplicial fan in
$N_\R$ whose cones are generated by all proper subsets of $\mb
a_1,\ldots,\mb a_{n+1}$. To make $m-n$ even we add one ghost
vector $\mb a_{n+2}$. Hence $m=n+2$, $\ell=1$, and we have one
more linear relation $\mu_1\mb a_1+\cdots+\mu_{n+1}\mb a_{n+1}+\mb
a_{n+2}=\mathbf0$, this time the $\mu_k$'s are arbitrary reals.

The subspace $\Ker A\subset\R^{n+2}$ is spanned by
$(\lambda_1,\ldots,\lambda_{n+1},0)$ and
$(\mu_1,\ldots,\mu_{n+1},\!1)$.

Then $\sK=\sK_\Sigma$ is the boundary of an $n$-dimensional
simplex with $n+1$ vertices and one ghost vertex, $\zk\cong
S^{2n+1}\times S^1$, and
$U(\sK)=(\C^{n+1}\setminus\{\bf0\})\times\C^\times$.

Conditions~(a) and~(b) of Construction~\ref{psi} imply that $C$ is
a 1-dimensional subgroup in $(\C^\times )^m$ given in appropriate
coordinates by
\[
  C=\bigl\{(e^{\zeta_1 w},\ldots,e^{\zeta_{n+1}w},e^w)
  \colon w\in\C\bigl\}\subset(\C^\times )^m,
\]
where $\zeta_k=\mu_k+\alpha\lambda_k$ for some
$\alpha\in\C\setminus\R$. By changing the basis of $\Ker A$ if
necessary, we may assume that $\alpha=i$. The moment-angle
manifold $\zk\cong S^{2n+1}\times S^1$ acquires a complex
structure as the quotient $U(\sK)/C$:
\begin{multline*}
  \bigl(\C^{n+1}\setminus\{\mathbf 0\}\bigr)\times\C^\times\bigl/\;
  \bigl\{(z_1,\ldots,z_{n+1},t)\!\sim
  (e^{\zeta_1w}z_1,\ldots,e^{\zeta_{n+1}w}z_{n+1},
  e^w t)\bigr\}
  \\
  \cong\bigl(\C^{n+1}\setminus\{\mathbf0\}\bigr)\bigl/\;
  \bigl\{(z_1,\ldots,z_{n+1})\!\sim
  (e^{2\pi i\zeta_1}z_1,\ldots,
  e^{2\pi i\zeta_{n+1}}z_{n+1})\bigr\},
\end{multline*}
where $\mb z\in\C^{n+1}\setminus\{\mathbf0\}$, $t\in\C^\times$.
The latter is the quotient of $\C^{n+1}\setminus\{\mathbf0\}$ by a
diagonalisable action of $\Z$. It is known as a \emph{Hopf
manifold}. For $n=0$ we obtain a complex 1-dimensional torus
(elliptic curve) of Example~\ref{2torus}.

Suppose we are in the situation of Theorem~\ref{onedim}~(a). Then
all $\lambda_k$ are commensurable (the ratio of each pair is
rational). We can assume that all $\lambda_k$ are integer
(multiplying them by a common factor if necessary). Then $\Sigma$
is a rational fan and $V_\Sigma=\C
P^n(\lambda_1,\dots,\lambda_{n+1})$ is a weighted projective
space, whose orbifold structure may have singularities in
codimension one. This gives rise to a holomorphic principal
Seifert bundle $\pi\colon\zk\to\C
P^n(\lambda_1,\dots,\lambda_{n+1})$ with fibre an elliptic curve.

Now suppose we are in the situation of Theorem~\ref{onedim}~(b).
For example, this is the case when
$\lambda_1,\ldots,\lambda_{m+1}$ are linearly independent
over~$\Q$. Then any submanifold of $\zk$ is a Hopf manifold of
lesser dimension (including elliptic curves and points).
\end{example}

\subsection{Divisors and meromorphic functions}
For simplicity, by a \emph{divisor} on a complex manifold we mean
an analytic subset of codimension one. (The description of
divisors in~$\zk$ obtained below can be easily modified to
match the more standard definition of divisors as
linear combinations of analytic subsets of codimension one.)

For generic initial data, there are only finitely many divisors on
the complex moment-angle manifold~$\zk$, and they are of a very
special type. This holds without any geometric restrictions on the
fan:

\begin{theorem}\label{nodiv}
Assume that the data $\{\sK;\mb a_1,\ldots,\mb a_m\}$ define a
complete fan~$\Sigma$ in~$N_\R\cong\R^n$, and $m-n=2\ell$. Assume
further that
\begin{itemize}
\item[(a)] there is at most one ghost vertex in~$\sK$;

\item[(b)] no rational linear function on $\R^m$ vanishes identically on $\Ker A$.
\end{itemize}
Then any divisor of $\zk$ is a union of coordinate divisors.
\end{theorem}
\begin{proof}
Let $\mathcal D\subset \zk$ be a divisor.  Consider the divisor
$q^{-1}(\mathcal D)$ in $U(\sK)$, where $q\colon U(\sK)\to
\zk\cong U(\sK)/C$ is the quotient projection.

First assume that there are no ghost vertices in~$\sK$, so that
all $\mb a_k$ are nonzero. Then $\C^m\setminus U(\sK)$ has
codimension $\ge 2$. Hence the closure of $q^{-1}(\mathcal D)$ in
$\C^m$ is a $C$-invariant divisor in~$\C^m$. Choose an element
$\mb u=(u_1,\ldots,u_m)$ in $C$ such that $|u_k|>1$,
$k=1,\ldots,m$. Such an element exists since there is a relation
$\sum_{k=1}^m \lambda_k\mb a_k=0$ with all $\lambda_k>0$ (this
follows from the fact that $\Sigma$ is a complete fan). Denote by
$L$ the discrete subgroup of $C$ consisting of integral powers
of~$\mb u$; then $L\cong\Z$. Being a subgroup of $C$, the group
$L$ is diagonalisable and acts freely and properly on
$\C^m\backslash\{\mathbf0\}$, so the quotient is a Hopf manifold.
By Example~\ref{hopf}, any Hopf manifold is a complex moment-angle
manifold with $\ell=1$. Then if follows from
Theorem~\ref{onedim}~(b) that any analytic subset of
$(\C^m\backslash \{\mathbf 0\})/L$ is a union of coordinate
submanifolds. Hence the closure of $q^{-1}(\mathcal D)/L$ in
$(\C^m\backslash \{\mathbf 0\})/L$ is a union of coordinate
divisors, and the same holds for $\mathcal D\subset\zk$.

Now assume that there is one ghost vertex in~$\sK$, say the first
one. Then $U(\sK)=\C^\times\times U(\widetilde\sK)$, where
$\widetilde\sK$ does not have ghost vertices. Since the divisor
$q^{-1}(\mathcal D)\subset\C^\times\times U(\widetilde\sK)$ is
$C$-invariant, its projection to the first factor $\C^\times$ is
onto. Therefore, for any $z_1\in\C^\times$, the intersection
$\bigl(\{z_1\}\times U(\widetilde\sK)\bigr)\cap q^{-1}(\mathcal
D)$ is a divisor in $\{z_1\}\times U(\widetilde\sK)$. This divisor
is invariant with respect to the subgroup $\widetilde
C=\{(u_1,u_2,\ldots,u_m)\in C\colon u_1=1\}$. Choose an element
$\mb u=(1,u_2,\ldots,u_m)$ in $\widetilde C$ such that $|u_i|>1$
for $i\ge2$. Such an element exist since now we have a relation
$\sum_{i\ge2} \mu_i\mb a_i=0$ with all $\mu_i>0$, by the
completeness of the fan. Now we proceed as in the case when there
are no ghost vertices, and conclude that each $\bigl(\{z_1\}\times
U(\widetilde\sK)\bigr)\cap q^{-1}(\mathcal D)$ is a union of
coordinate divisors in $\{z_1\}\times U(\widetilde\sK)$. Since the
number of coordinate divisors is finite, $\bigl(\{z_1\}\times
U(\widetilde\sK)\bigr)\cap q^{-1}(\mathcal
D)=\{z_1\}\times\mathcal E$, where $\mathcal E\subset
U(\widetilde\sK)$ is a union of coordinate divisors. Thus,
$q^{-1}(\mathcal D)=\C^\times\times\mathcal E$, and $\mathcal D$
also has the required form.
\end{proof}

\begin{corollary}\label{nomerfun}
Under the assumptions of Theorem~{\rm\ref{nodiv}}, there are no
non-constant meromorphic functions on~$\zk$.
\end{corollary}
\begin{proof}
Let $f$ be a non-constant meromorphic function on~$\zk$. Choose a
point $\mb z_0\in\zk$ in the dense $(\C^\times)^m/C$-orbit outside
of the pole set of $f$. Then the support of the zero divisor of
the function $f(\!\mb z)-f(\!\mb z_0)$ contains a point $\mb z_0$
in the dense $(\C^\times)^m/C$-orbit, so it does not lie in the
union of coordinate divisors. This contradicts
Theorem~\ref{nodiv}.
\end{proof}

\subsection{General subvarieties}\label{gensubm}
As we can see from Theorem~\ref{onedim}, the geometry of $\zk$
depends essentially on the geometric data, namely on a choice of
maps~$A$ and~$\varPsi$. In the situation of Theorem~\ref{onedim}
(i.e. $\ell=1$), the case when no rational function vanishes on
$\Ker A$ is generic; so that $\zk$ has only coordinate
submanifolds for generic geometric data. As we shall see, the
situation is similar in the case of higher-dimensional foliations
($\ell>1$), although the generic condition on the initial data
will be more subtle.

\begin{lemma}\label{generic_cond}
Assume that data $\{\sK;\mb a_1,\ldots,\mb a_m\}$ define a
complete fan $\Sigma$ in $N_\R\cong \R^n$ such that no rational
linear function on $\R^m$ vanishes on $\Ker A$. Then for almost
all subspaces $\mathfrak{c} \subset \C^m$ satisfying the
conditions of Construction~{\rm\ref{psi}}, the only complex
subspace~$L\subset \C^m$ such that
\begin{itemize}
\item[(a)] $\mathfrak{c}\subset L$,
\item[(b)] $\overline{\mathfrak{c}}\cap L\ne\{\bf0\}$,
\item[(c)] $L\cap i\R^m$ is a rational subspace in $i\R^m$,
\end{itemize}
is the whole $L=\C^m$.
\end{lemma}

\begin{proof}
The subspace $\mathfrak{c}=\varPsi(\C^\ell)$ is a point in the
Grassmannian $\Gr_\C(\C^\ell, \Ker A_\C)$, which has complex
dimension $\ell^2$. The conditions of Construction~\ref{psi}
specify an open subset in $\Gr_\C(\C^\ell, \Ker A_\C)$; we refer
to a map $\varPsi\colon\C^\ell\to\C^m$ satisfying these conditions
and the corresponding subspace $\mathfrak{c}=\varPsi(\C^\ell)$ as
\emph{admissible}. We shall prove that the set of admissible
$\mathfrak{c}$ for which there exist $L\subsetneq\C^m$ satisfying
properties (a)--(c) is contained in a countable union of manifolds
of dimension $<\ell^2$ and therefore has zero Lebesgue measure.

Let $L\subsetneq\C^m$ be a complex subspace satisfying conditions
(a)--(c). Set $Q=\Ker A\cap L$, and let $\dim_\R Q=q$. Conditions
(a) and (b) imply $q>0$. Also, since no rational linear function
vanishes on $\Ker A$ and $L\cap i\R^m$ is a proper rational
subspace, $L$ cannot contain the whole~$\Ker A$. Hence,
$0<q<2\ell$. Let $\pi_{\Re},\pi_{\Im}\colon \mathfrak{c}\to \Ker
A$ denote the projections onto the real and imaginary parts (which
are both isomorphisms of real spaces). For a given $v\in Q$ there
exists a unique $w\in\Ker A$ such that $v+iw\in\mathfrak{c}\subset
L$. Since $v,v+iw\in L$, the vector $w$ lies in $L$. Therefore,
\begin{equation}\label{q_proj}
  \pi_{\Im}\circ\pi^{-1}_{\Re}(Q)=Q,
\end{equation}
as the operator on the left hand side sends $v$ to $w$. In
particular, the operator $\pi_{\Im}\circ\pi^{-1}_{\Re}$ defines a
complex structure on the space~$Q$. Hence $\dim_\C
\mathfrak{c}\cap (Q\otimes\C)=q/2$. Now, for a fixed subspace
$Q\subsetneq \Ker A$, the set of complex subspaces
$\mathfrak{c}\subset\Ker A_\C$ satisfying condition~\eqref{q_proj}
is identified with an open subset in
\[
  \Gr\nolimits_\C\bigl(\C^{q/2},Q\otimes\C\bigr)\times
  \Gr\nolimits_\C\bigl(\C^{\ell-q/2}, (\Ker
  A_\C)/(Q\otimes\C)\bigr)
\]
and has complex dimension $(q/2)^2+(\ell-q/2)^2<\ell^2$. Since
there are only countably many rational subspaces in $i\R^m$, there
are countably many $Q\subset \Ker A$. Hence the set of spaces
$\mathfrak{c}$ for which there exist $L\subsetneq\C^m$ satisfying
properties (a)--(c) is contained in the union of countably many
manifolds of dimension $<\ell^2$. Thus its Lebesgue measure is
zero.
\end{proof}

Our final results describe analytic subsets in a complex
moment-angle manifold $\zk$, under a generic assumption on the
complex structure and a geometric assumption on the fan $\Sigma$:

\begin{theorem}\label{final}
Let $\zk$ be a complex moment-angle manifold, with linear maps
$A\colon\R^m\to N_\R$ and $\varPsi\colon\C^\ell\to\C^m$ described
in Construction~{\rm\ref{psi}}. Assume that
\begin{itemize}
\item[(a)]no rational linear function on
$\R^m$ vanishes on $\Ker A$;
\item[(b)] $\Ker A$ does not contain rational vectors of~$\R^m$;
\item[(c)]the map $\varPsi$ satisfies the generic condition of
Lemma~{\rm\ref{generic_cond}};
\item[(d)] the fan $\Sigma$ is weakly normal.
\end{itemize}
Then any irreducible analytic subset $Y\subsetneq\zk$ of
positive dimension is contained in a coordinate submanifold $\mathcal Z_{\mathcal K_J}\subsetneq \zk$.
\end{theorem}
\begin{proof}
Assume that $Y\subset\zk$ is an irreducible analytic subset of the
smallest positive dimension. Let $\cF_Y$ be the foliation on $Y$
associated with $\cF$, i.e. $T_y\cF_Y=T_yY \cap T_y\cF$ (here
$T_yY$ denotes the Zariski tangent space). Assuming that $Y$
contains a generic point (i.e. a point from
$(\C^\times)^m/C\subset\zk$), we need to show that $Y$ is the
whole~$\zk$. As we assume~(d), Theorem~\ref{form2} applies, giving
a transverse K\"ahler form~$\omega_\cF$. Since $\omega_\cF$ is
exact, the integral $\int_Y \omega_\cF^{\dim Y}$ vanishes, hence
the foliation $\cF_Y$ is nontrivial. For any $z\in\zk$, the
tangent space $T_y \cF$ is naturally identified with the vector
space $\Ker A_\C/\mathfrak{c}$. Let $k$ be the complex dimension
of $\cF_Y$ at a generic point of~$Y$, and let $\widetilde Y\subset
Y\times \mathop{\mathrm{Gr}}_\C(\C^k,\Ker A_\C/\mathfrak{c})$ be
the space
of all $k$-dimensional planes $V\subset T_yY \cap T_y\cF$. %
Denote by $\pi_Y$ and $\pi_G$ the projections of $\widetilde Y$ to
$Y$ and $\mathop{\mathrm{Gr}}_\C(\C^k,\Ker A_\C/\mathfrak{c})$,
respectively. For any $k$-dimensional plane $V\subset \Ker
A_\C/\mathfrak{c}$, the analytic subset $\pi_Y(\pi_G^{-1}(V))$ is
identified with the closure of the set of all points $y\in Y$ such
that $T_y Y=V$. Since $Y\subset\zk$ is an analytic subset of the
smallest dimension, $\pi_Y(\pi_G^{-1}(V))$ either is 0-dimensional
for all $V$ or coincides with $Y$.

Assume that $\dim\pi_Y(\pi_G^{-1}(V))=0$ for all $V\subset \Ker
A_\C/\mathfrak{c}$. Then $\widetilde Y$ admits a meromorphic map
to $\mathop{\mathrm{Gr}}(\C^k,\Ker A_\C/\mathfrak{c})$ which is
finite at a generic point, hence $\widetilde Y$ is Moishezon. The
map $\widetilde Y\to Y$ is surjective, so $Y$ is also Moishezon.
Then $Y$ is of Fujiki class $\mathcal C$ by the classical result
of~\cite{mois67}.
By the assumption~(b), Corollary~\ref{leavescor} applies, leading
to a contradiction.

Now assume that $\pi_Y(\pi_G^{-1}(V))=Y$ for some $k$-dimensional
plane $V\subset \Ker A_\C/\mathfrak{c}$. In other words, $T_yY
\cap T_y\cF=V$ for a generic point $y\in Y$. Let $H\subset
(\C^\times)^m/C$ be the largest closed complex subgroup preserving
$Y\subset U(\sK)/C$, and let $\mathfrak{h}\subset
\C^m/\mathfrak{c}$ be the Lie algebra of~$H$. Let $L\subset \C^m$
be the preimage of $\mathfrak{h}$. Then
\begin{itemize}
\item[--] $\mathfrak{c}\subset L$;
\item[--] $L\cap i\R^m$ is a rational subspace, since $H$ is closed;
\item[--] $\overline{\mathfrak{c}}\cap L\ne\{\bf0\}$, since $\mathfrak{h}\supset V\subset
\Ker A_\C/\mathfrak c$.
\end{itemize}
As we assume (a) and (c), Lemma~\ref{generic_cond} applies,
implying that $L=\C^m$. Hence, $H=(\C^\times)^m/C$ and $Y=\zk$.
\end{proof}

\begin{corollary}\label{final2}
Assume that every coordinate submanifold $\mathcal Z_{\mathcal
K_J}\subseteq \zk$ either satisfies the assumptions {\rm(a)--(d)}
of Theorem~{\rm\ref{final}}, or is a compact complex torus with no
analytic subsets of positive dimension. Then any irreducible
analytic subset $Y\subsetneq\zk$ of positive dimension is a
coordinate submanifold $\mathcal Z_{\mathcal K_J}\subsetneq \zk$.
\end{corollary}
\begin{proof}
Let $Y\subsetneq \zk$ be an irreducible analytic subset. Applying
Theorem~\ref{final} to $\zk$ we conclude that $Y$ is contained in
a coordinate submanifold $\mathcal Z_{\mathcal K_J}\subsetneq
\zk$. If $Y=\mathcal Z_{\mathcal K_J}$, we are done. Otherwise we
have two options: either $\mathcal Z_{\mathcal K_J}$ itself
satisfies the assumptions of Theorem~\ref{final} and we can
proceed by induction, or $\mathcal Z_{\mathcal K_J}$ is a compact
complex torus. In the latter case $Y=\mathcal Z_{\mathcal K_J}$,
by the assumption.
\end{proof}

\end{document}